\documentclass[reqno]{amsart}
\usepackage[utf8]{inputenc}
\usepackage{alphabeta}

\usepackage{personal-commands} 
\usepackage{amsthm}
\usepackage{amssymb}
\usepackage{mathrsfs}
\usepackage{cancel}
\usepackage{hyperref}

\usepackage[left=3.5cm, right=3.5cm]{geometry}
\calclayout 
\usepackage{mathtools}
\usepackage{adjustbox}

\newtheorem{theorem}{Theorem}[section]
\newtheorem{lemma}[theorem]{Lemma}
\newtheorem{corollary}[theorem]{Corollary}
\newtheorem{proposition}[theorem]{Proposition}

\theoremstyle{definition}
\newtheorem{remark}[theorem]{Remark}

\newtheorem{problem}[theorem]{Problem}
\newtheorem{question}[theorem]{Question}
\newtheorem{definition}[theorem]{Definition}
\newtheorem{examplex}[theorem]{Example}
\newenvironment{ex}
  {\pushQED{\qed}\examplex}
  {\popQED\endexamplex}

\newtheorem{example}[theorem]{Example}

\usepackage[
			backend=biber, 
			natbib=true,
			bibencoding=utf8,
			style=alphabetic,
			labelalpha=true,
			maxalphanames=5, 
			minalphanames=2,	
]{biblatex}  
\AtEveryBibitem{
    \clearfield{urlyear}
    \clearfield{urlmonth}
}

\usepackage{xcolor}

\author{Tommaso Faustini \and Alejandro Vargas}

\begin{document}


\title{Almost-valuative invariants of connected split matroids: The \texttt{cd}-index}

\begin{abstract}
    We derive a formula for matroid invariants $\Psi$ on a large family of matroids, 
    provided that $\Psi$ is almost-valuative, 
    namely, it satisfies a hyperplane-cut formula. 
    Our primary application is to the \texttt{cd}-index $\Psi_{cd}$ of the base polytope $\scrP(M)$, 
    a polynomial in two non-commutative variables that compactly encodes the number of face-flags $\calF = \{\sigma_1 \subset \dots \subset \sigma_s \}$ with prescribed dimensions $\dim \sigma_i = d_i$. 
    This generalizes recent work by Ferroni and Schröter on the $f$-vector of $\scrP(M)$,
    yielding a formula that can be understood as a valuative part plus an error term that surprisingly depends only on modular pairs of cyclic flats.
   This enables computations requiring only the following data: 
   the evaluations of $\Psi$ on hypersimplices $\Delta_{k,n}$ 
   and cuspidal matroids $\Lambda^{r,h}_{k,n}$;
   and counts $\lambda(r,h)$ and $\mu(a,b;\alpha,\beta)$ 
   of cyclic flats and modular pairs of cyclic flats in $M$, 
   respectively,
satisfying specific rank and cardinality conditions.
   We compute these for the \texttt{cd}-index, 
   yielding explicit results for sparse paving matroids and rank-2 matroids.
\end{abstract}


\maketitle

\section{Introduction}
    \label{sec:Introduction}

    Matroid invariants are central to recent breakthroughs in combinatorics and combinatorial algebraic geometry, 
    particularly in establishing the real-rootedness, log-concavity, and unimodality of various combinatorial sequences;
    for some introductory surveys see \cite{ard18, kal22, el24}.  
    Many of these invariants $\Psi$ satisfy an inclusion-exclusion property with respect to matroidal subdivisions of the base polytope $\scrP(M)$;
    i.e.~for a hyperplane $H$ that cuts $\scrP(M)$ into two matroid polytopes we have 
    $\Psi(M) = \Psi(\scrP(M) \cap H^+) + \Psi(\scrP(M) \cap H^-) - \Psi(\scrP(M) \cap H)$. 
    This \emph{valuative property} is remarkable; 
    it is rarely evident from the invariants' definitions, 
    which are typically algebraic rather than combinatorial in nature. 

    Exploiting this valuative property, Ferroni and Schröter \cite{fs24c} introduced a framework to efficiently compute $\Psi(M)$ for a broad class of matroids $M$ known as \emph{split matroids}: 
    the value $\Psi(M)$ is determined by the value of $\Psi$ on hypersimplices and so-called \emph{cuspidal matroids}.
    This is significant because split matroids encompass \emph{paving matroids}, 
    and a well-known conjecture of Crapo and Rota, see \cite{mnww11}, posits that paving matroids dominate asymptotically:
    for a fixed rank $k$, the proportion of paving matroids among all matroids on $n$ elements approaches 1 as $n \to \infty$. 

    Split matroids initially emerged in tropical geometry~\cite{js17}.     
    Geometrically, they are obtained from the hypersimplex $\Delta_{k,n}$ by making multiple cuts 
    that are mutually non-interfering, 
    i.e.~no point of $\Delta_{k,n}$ is eliminated twice.
    These cuts correspond to cyclic flats, and any pair of them intersects inside $\Delta_{k,n}$ only for modular pairs,
    i.e.~flats $F,G$ such that $\rk(F) + \rk(G) = \rk(F \cup G) + \rk(F \cap G)$.

    While the Ferroni-Schröter framework 
    provides effective tools to gather evidence or find counterexamples to conjectures
    dealing with valuative invariants, not all important invariants are valuative.
    A natural next step is to study the failure of valuativeness and find new strategies for computation.
    For example, consider the $f$-vector, a fundamental invariant in polyhedral geometry that counts the number of faces in each dimension of a polytope $P$.
    Although the $f$-vector of a matroid polytope fails to be valuative, 
    the main result of~\cite{fs25} demonstrates that its calculation can still be expressed in terms of simplices and cuspidal matroids, with an expression similar to the valuative setting plus an error term. 
    The latter is computationally tractable,
    arising precisely from the double intersections of the aforementioned cuts, 
    thus expressed as a sum over modular pairs of cyclic flats.
    Surprisingly, while triple or higher-order intersections of these cuts are possible, they carry no structural data that contributes to the error term. 
    
    A natural generalization of the $f$-vector is the \emph{flag $f$-vector}, 
    whose entries count the number of face-flags $\calF = \{\sigma_1 \subset \dots \subset \sigma_s \}$ of prescribed dimensions $\dim \sigma_i = d_i$. 
    Flag $f$-vectors satisfy the \emph{generalized Dehn-Sommerville relations} \cite{bb83}.
    These relations imply that the flag $f$-vector can be compactly encoded into a non-commutative polynomial in two variables known as the \texttt{cd}-index \cite{bk91}.
    Our initial aim was to extend the methods of \cite{fs25} to the \texttt{cd}-index,
    to obtain an evaluation that depends only on the \texttt{cd}-indices of hypersimplices and cuspidal matroids.
    It emerged, once we were finished with the calculations, 
    that the key ingredient is a hyperplane-cut formula due to Kim~\cite{kim10}, 
    expressing a property that we call being \emph{almost-valuative}.
    Any almost-valuative invariant $\Psi$ can be calculated with our results. 
    In particular, this \emph{almost-valuative} framework naturally recovers  the results of \cite{fs25} on the $f$-vector, 
    because the $f$-vector satisfies a rather self-evident hyperplane-cut formula.

\subsection*{Results and structure of the paper}
In Section~\ref{sec:CountingChainsOfFaces} we recall the definition of the \texttt{cd}-index and several results aiding its calculation,
including Kim's hyperplane-cut formula.
In essence, this formula says that the correction preventing $\Psi_{cd}(\scrP)$ from being valuative is a function $\calE$ depending entirely on the polytopes of the form $\sigma \cap H$, where $\sigma$ is a face of $\scrP$ intersected by $H$ in its interior.
In Proposition~\ref{prop:PhiForHypersimplex} we derive a recursive formula for the hypersimplex $\Delta_{k,n}$ using a description of faces via deletion and contraction facets, and clever manipulations that depend on the highly symmetric nature of $\Delta_{k,n}$.

In Section~\ref{sec:Matroids} we recall some matroid facts, 
starting with how the matroid base polytope $\scrP(M)$ can be obtained by chopping off  pieces of the hypersimplex $\Delta_{k,n}$, 
and that the cuts $H_F$ are in correspondence with proper cyclic flats $F \in \calZ(M)$. 
We recall from \cite{fs24c} the theory of relaxation of stressed subsets and that a connected matroid $M$ is split if and only if every pair $F\neq G \in \properZ(M)$ is incomparable.
A cuspidal matroid $\Lambda^{r,F}_{k,n}$ is a connected matroid with $\calZ(M) = \aset{\varnothing, F, E}$;
the parameters $r$, $k$ and $n$ denote the rank of $F$, the rank of $\Lambda$, and the cardinality of groundset $E$. 
We characterize the faces of $\Lambda^{r,F}_{k,n}$ in Proposition~\ref{prop:FacesOfCuspidal} and with that give a recursive formula for the \texttt{cd}-index in Proposition~\ref{lm:PhiForCuspidal}.

Section~\ref{sec:TheGeometryOfSplitConnectedMatroids} is dedicated to the geometry of connected split matroids. 
We recall that the cuts $H_F$ and $H_G$ intersect inside $\Delta_{k,n}$ if and only if $F$ and $G$ are a modular pair;
that any vertex of $\Delta_{k,n}$ not present in $\scrP(M)$ is eliminated by exactly one cut $H_F$;
that the face corresponding to a flat $G$ remains unchanged if we relax by any $F \in \properZ(M)$ distinct from $G$;
that for a modular pair $(F,r)$ and $(G,s)$ the polytope $Q = \Delta_{k,n} \cap H_F \cap H_G$ is isomorphic to the product of simplices $\Delta_{r - \gamma, F \setminus G} 
            \times \Delta_{s - \gamma, G \setminus F}$ with $\gamma = \card{F \cap G}$. 
In Proposition~\ref{prop:NoTripleIntersection} we show that in a connected split matroid, triple and higher intersections of cuts $H_F$ avoid the kind of polytopes seen in the correction term of Kim's formula.
This is the key geometric insight to our results.

Section~\ref{sec:FormulasForAlmostValuativeInvariants} begins deriving,
via induction on $\card{\properZ(M)}$, 
a formula for the \texttt{cd}-index $\Psi_{cd}$ of connected split matroids; see Equation~\eqref{eq:CDIndexSplitMatroid}.
The formula has a valuative part expressed in terms of hypersimplices and cuspidal matroids, and an error term governed by modular pairs of cyclic flats; see Equation~\eqref{eq:error}.
Subsection~\ref{sub:ComputationalData} makes several computational observations based on the fact that the value of $\Psi_{cd}$ is the same for all isomorphic matroids, 
and yields in Theorem~\ref{thm:RecursiveCD} a formula in terms of the number $\lambda(r,h)$ of proper non-empty cyclic flats of rank $r$ and size $h$,
and the number $\mu(a,b;\alpha,\beta)$ of modular pairs of cyclic flats $F, G$ satisfying
$a = \card{F \setminus G}$, $b = \card{G \setminus F}$, $\alpha = \rk(F) - \card{F \cap G} $ and $\beta = \rk(G) - \card{F \cap G} $.
In Subsection~\ref{sub:AlmostValuative} we introduce a definition for \emph{almost-valuative} functions $\Psi : \operatorname{MatPoly} \to \mathcal{G}$, 
with $(\mathcal{G}, +, \cdot)$ an arbitrary ring.
Namely, $\Psi$ is almost-valuative if for some function $\calE : \operatorname{MatPoly} \to \mathcal{G}$ and constant $K \in \mathcal{G}$ we have:
for all matroids $M$ and hyperplanes $H$ splitting $\scrP(M)$ into two matroid polytopes
\begin{align*}
\Psi(M) =& \Psi(\scrP(M) \cap H^+) + \Psi(\scrP(M) \cap H^-) - \Psi(\scrP(M) \cap H) \cdot K 
                 \nonumber &- \sum_{\sigma \in \calT_H(M)} \calE(\sigma \cap H).
\end{align*} 
Here $\calT_H(M)$ is the set of faces of $\scrP(M)$ with an interior point contained in $H$.
Our formulas for $\Psi_{cd}$ generalize straightforwardly to almost-valuative invariants $\Psi$, 
yielding our main theorem:
\begin{theorem}  
    \label{thm:AlmostValuativeInvariantsIntro}
    Let $M$ be a connected split matroid such that $\properZ(M) = \aset{(F_i,r_i)}_{i \in [m]}$, 
    and $\Psi$ an almost-valuative invariant with correction~$\calE$.
    We have that
    \[ \Psi(M) = \Psi(\Delta_{k,n}) + \sum_{F \in \properZ(M)} (\Psi(\Lambda^F_M) - \Psi(\Delta_{k,n}) )- \sum_{(F,G) \in \calM \calP(M)} W_{F,G}.\]
Here $\calM \calP(M)$ is the set of modular pairs $(F_i, F_j)$ of cyclic flats with $i < j$, and
\begin{align}
    W_{F,G}=\sum_{\sigma \in \calS(F,G)} \calE(\sigma \cap H).
\end{align}
The sum is over the set $\calS(F,G)$ of faces $\sigma$ of $\calT_{H_G}(\Delta_{k,n})$ such that $\sigma \cap H_G \subset H_F$.

Moreover, we have that $W_{F,G}$ depends only on $a = \card{F \setminus G}$, $b = \card{G \setminus F}$, $\alpha = \rk(F) - \card{F \cap G} $ and $\beta = \rk(G) - \card{F \cap G} $. Also
\begin{align*} 
\Psi(M) =& \Psi(\Delta_{k,n}) 
    + \sum_{0 < r < h < n} \lambda(r,h) \left( \Psi(\Lambda^{r,h}_{k,n}) - \Psi(\Delta_{k,n}) \right) \\
&- \sum_{\mbfr} \mu(\mbfr) \cdot W(\mbfr).
\end{align*}
The sum ranges over all $\mbfr = (a,b; \alpha,\beta)$ 
such that $0 < \alpha < a < n$, $0 < \beta < b < n$ and $(\alpha,a)\le_{\operatorname{lex}} (\beta,b)$ in the lexicographic order.

\end{theorem}
We show that the $f$-vector fits in this framework and recover the result of \cite{fs25} in Subsection~\ref{sub:AlmostValuative}.

Finally, in Section~\ref{sec:ApplicationsAndFutureDirections} we present applications to sparse paving matroids, 
computational code, 
and several questions to outline future directions. 

\subsection*{Notation}
Throughout this article we use 
\begin{itemize} 
    \item $E$ a finite set, $\RR^E = \aset{f : E \to \RR}$  with standard basis $\aset{\mbfe_i}_{i \in E}$ and inner product $\langle -, - \rangle$.
\item $k$ and $n$ the rank and elements of a matroid $M$ on $E$. 
\item For $F \subset E$, we denote $E \setminus F$ by $F^*$,  $h = \card{F}$,  $r = \rk(F)$, $r^* = k - r$,  $h^* = n - h$.
\item  $H^+(u,m) = \aset{x \in \RR^E \suchthat \langle u, x \rangle \ge  m }$ 
    and $ H^-(u,m) =  \aset{x \in \RR^E \suchthat \langle u, x \rangle \le  m }$
    for closed half-spaces,
    and $ H(u,m) =  H^+(u,m) \cap H^-(u,m)   $ for a hyperplane.
  \item $\calT_{H(u,m)}(P)$ for the set of all proper faces $\sigma$ of a polytope $P$ intersecting the interior of both halfspaces  $H^+(u, m)$  and  $H^-(u, m)$.
\item $(F,r) \in \properZ(M)$ for a proper cyclic flat of rank $r$ of a matroid  $M$.
\item $H_F$ for $H(u_F, r)$ and $\calT_F(M)$ for $\calT_{H_F}(\scrP(M))$,
with $\mbfe_F =  \sum_{i \in F} \mbfe_i$ the indicator vector for $F$ and $u_F = \langle \mbfe_F, - \rangle$ the functional that defines the cut associated to $(F,r)$.
\end{itemize}

\section{Counting chains of faces}
    \label{sec:CountingChainsOfFaces}

\subsection{\texorpdfstring{Flag $f$-vectors and the \texttt{cd}-index}{Flag f-vectors and the cd-index}}
    \label{sub:FlagCdIndex}
We assume that the reader is acquainted with some basic polytope theory, 
e.g.~first two chapters of \cite{zie12}.
Given a $d$-dimensional polytope $P \subset \RR^n$ 
the \emph{flag $f$-vector of $P$}  is indexed by the power set $\calP(\aset{0, \dots, d})$;
it has entries $f_S$ 
counting the number of face-flags $\sigma_1 \subset \dots \subset \sigma_s$ 
such that $\aset{\dim \sigma_1, \dots, \dim \sigma_s}$ equals $S$.
For example, $f_\varnothing = 1$ counts the empty chain, 
$f_{\aset{0}}$ counts vertices of $P$, 
$f_{\aset{0,d-1}}$ counts pairs $(v, \sigma)$ of vertices and facets with $v \in \sigma$.
The usual $f$-vector corresponds to the entries indexed by singletons: 
$(f_{\aset{0}}, f_{\aset{1}}, \dots, f_{\aset{d}})$. From now on, we assume a flag always starts with the empty set and finishes with $P$.
Consider the following polynomial in non-commuting variables  $a$ and $b$, called the \texttt{ab}-index:
\begin{align}
\label{eq:AB-index}
\Psi_{ab}(P) = \sum_{\calF \text{ face-flag}} w(\calF),
\end{align}
where the weight of a flag $\varnothing \subset \sigma_1 \subset \dots \subset \sigma_s \subset P$ is $w(\calF) = w_0 \dots w_{d-1}$ with $w_i = b$ if $i \in \aset{\dim \sigma_1, \dots, \dim \sigma_s}$, and $a-b$ otherwise.

Note that $\Psi_{ab}(P)$ is homogeneous of degree $d =\dim P$ for $\deg a = \deg b = 1$.
While the space of non-commutative degree-$d$ polynomials on $a,b$ has dimension $2^d$, 
those appearing as \texttt{ab}-indices are governed by the \emph{generalized Dehn-Sommerville relations} \cite{bb83}.
These are linear relations among the coefficients, characterized by:
\begin{proposition}[Thm~4 of \cite{bk91}]
    \label{prop:CDindexBayer}
Let $\Pi$ be a finite poset.
There exists an integer-coefficient polynomial $\Psi_{cd}(\Pi)$
such that the substitution $c = a + b$ and $d = ab + ba$ yields $\Psi_{ab}(\Pi)$ 
if and only if $\Pi$ satisfies the generalized Dehn-Sommerville relations.
\end{proposition}
All finite Eulerian posets satisfy Proposition~\ref{prop:CDindexBayer},
where, instead of face-flags, chains are used to define $\Psi_{ab}$.
The face-poset of a polytope satisfies the Eulerian condition.
Note that the grading $\deg c = 1$ and $\deg d = 2$ yields an homogeneous polynomial $\Psi_{cd}$,
and that the standard basis for the degree-$d$ part has Fibonacci-many terms.

\subsection{Cuts and products}
\label{sub:CutsAndProducts}
There are several techniques to compute the \texttt{cd}-\emph{index} $\Psi_{cd}$ of a polytope, see \cite{bay21} for a survey.
We use Kim's formula for hyperplane splits~\cite{kim10}.
For a polytope $P \subset \RR^n$ and 
a hyperplane 
$H(u, m) = \aset{x \in \RR^n \suchthat \langle u, x \rangle = m}\subset \RR^n$ 
intersecting the interior of $P$, we have: 
\begin{align}
    \label{eq:KimsFormula}
\Psi_{cd}(P) =& \Psi_{cd}(P^+) + \Psi_{cd}(P^-) - \Psi_{cd}(P \cap H) \cdot c 
\\ \nonumber &- \sum_{\sigma \in \calT_{H(u,m)}(P)} \Psi_{cd}(\sigma \cap H) \cdot d \cdot \Psi_{cd}(P\cap H / (\sigma \cap H)).
\end{align} 
Here $\calT_{H(u,m)}(P)$ denotes the set of all proper faces $\sigma$ of $P$ intersecting the interior of both halfspaces  $H^+(u, m)$  and  $H^-(u, m)$ 
non-trivially (we chose $\calT$ because in Greek cut \emph{τομή}). 

In our investigations, the faces $\sigma$ appearing in Equation~\eqref{eq:KimsFormula} are products of polytopes.
Therefore, we use Ehrenborg, Fox and Readdy's formula for products of two polytopes $V$ and $W$ \cite[Prop. 6.2 and Thm. 7.1]{ef03}:
\begin{align} 
    \label{eq:ProductFormula}
  \Psi_{cd}(V \times W) = N(\Psi_{cd}(V), \Psi_{cd}(W)),
\end{align}
where $N$ is a bilinear operator satisfying a recursive formula. Thanks to this formula, in this exposition, we will focus on the case of connected matroids.

\subsection{A mixed expression}
\label{sub:AMixedExpression}
Grouping together chains by their last element gives us a mixed expression in variables $a, b, c, d$. 
We say that two polynomials $f$ and $g$ in these variables are equivalent, denoted $f \sim g$, if the substitution $c = a + b$,  $d = ab + ba$ yields two equal polynomials.

\begin{lemma} 
    \label{lm:CDIndexStratified}
    Let $P$ be a polytope, $\calF(P)$ its face lattice, and $\vertici P$ its vertices.
    The \texttt{cd}-index $\Psi_{cd}(P)$ is equivalent to
    \begin{align} 
        \label{eq:StratifyCDIndex}
        \Phi(P) = (a-b)^{\dim P} + \card{\vertici P} b(a-b)^{\dim P - 1} +  \sum_{\sigma \in \calF_{\geq 1}(P)} \Psi_{cd}(\sigma) b (a-b)^{\codim \sigma -1}.
    \end{align}
    Here $\codim \sigma$ is relative to $\affspan P$.
\end{lemma}

It is convenient to set apart the expressions corresponding to the empty chain and vertices; see Lemma~\ref{lm:FacesOfDeltakn}. 
Therefore, we set
\begin{align} 
    \label{eq:EmptyVertices}
    \EmVe(P) :=  (a-b)^{\dim P} + \card{\vertici P} b(a-b)^{\dim P - 1}
\end{align}

As our variables are non-commuting, 
we seek algebraic expressions to manipulate the right hand side of Equation~\eqref{eq:StratifyCDIndex} into an expression only in $c$ and $d$.

\begin{lemma}
    \label{lm:AlgebraForABCD}
Let $[-,-]$ denote the commutator.
The following relations between the four non-commutative variables $a$, $b$, $c=a+b$ and $d=ab+ba$ hold:
\begin{align*} 
    (a-b)^2 &=c^2-2d &
    b(a-b) &=d-cb\\
    a-b &=c-2b &
    \left[ b,(a-b)^2\right] &=\left[d,c\right].
\end{align*}
\end{lemma}
\begin{proof} 
The first three are a direct substitution. 
For the last:
\begin{align*} 
    \left[ b,(a-b)^2\right] &=  b(a-b)^2 - (a-b)^2b   \\
                            &=  ba^2 -\bcancel{bab}-b^2a+\bcancel{b^3}-a^2b+ab^2+\bcancel{bab}-\bcancel{b^3}\\
    \left[d,c\right] &= dc - cd = (ab+ba)(a+b) - (a+b) (ab+ba) \\
                            &=  \bcancel{aba}+ba^2+ab^2+\bcancel{bab} - a^2b -\bcancel{aba} -\bcancel{bab}-b^2a. \qedhere 
\end{align*}
\end{proof}

\begin{corollary} 
    \label{cor:}
  For all $m\in \NN$ we have that:
\begin{align*}
(a-b)^{2m}&\sim (c^2-2d)^{m}  , \\
(a-b)^{2m+1}&\sim (c^2-2d)^{m}(c-2b) , \\
b(a-b)^{2m} &\sim g_{cd}(2m) \\
b(a-b)^{2m+1} &\sim g_{cd}(2m+1). 
\end{align*} 
Here $g_{cd}$ is defined by the following recursion: 
{
\footnotesize{
\[
g_{cd}(t) = 
\begin{cases}
(c^2-2d)g_{cd}(t-2) + (dc-cd)(c^2-2d)^{\frac{t-2}{2}} & \text{if } t\geq2 \text{ even}, \\
(c^2-2d)g_{cd}(t-2) + (dc-cd)(c^2-2d)^{\frac{t-3}{2}}(c-2b)  & \text{if } t\geq2 \text{ odd}, \\
d - cb & \text{if } t=1, \\
b & \text{if } t=0.
\end{cases}
\]}
}
\end{corollary}

  \begin{proof}
    This is a direct consequence of Lemma \ref{lm:AlgebraForABCD}.
\end{proof}

The upshot is that we apply the four identities from Lemma~\ref{lm:AlgebraForABCD} to Equation~\eqref{eq:StratifyCDIndex} 
to eliminate the variable $a$ and get an expression of the form $p(c,d) + q(c,d) b$.
See Example~\ref{ex:U25} on Page~\pageref{ex:U25} for a concrete example, 
which also motivates us to observe: 

\begin{lemma} 
    \label{lm:vanishingb}
    Let $a,b,c,d$ be non-commutative variables, and $p,q,h$ polynomials in $c,d$.
    If
    \[ 
        p(c,d) + q(c,d)b = h(c,d), 
    \]
    then $q$ is identically zero.
\end{lemma}

\begin{proof} 
Both $c$ and $d$ are symmetric under the exchange of $a$ and $b$. 
Thus, so is any equation in $c$ and $d$, in particular $h(c,d)-p(c,d) = q(c,d)b$. 
But $q(c,d)b$ is non-symmetric if $q(c,d) \neq 0$. 
The result follows.
\end{proof}

\subsection{\texorpdfstring{The \texttt{cd}-index of the hypersimplex $\Delta_{k,n}$}{The cd-index of the hypersimplex Delta(k,n)}}
\label{sub:TheCDIndexOfAHypersimplex}
For a finite set $E$ with $n$ elements and $k \in \aset{0, \dots, \card{E}}$ 
the hypersimplex $\Delta_{k,E}$ is the convex hull in $\RR^{E}$ 
of the indicator vectors $\mbfe_B = \sum_{i \in B} \mbfe_i$ for all the $k$-subsets $B \subset E$. 
We write $\Delta_{k,n}$ for $\Delta_{k,\aset{1, \dots, n}}$.
For $k=0$ or $k=n$ we get the point $\mathbf 0$ or $\mathbf 1$, respectively, and for $k=1$ or $k=n-1$ the simplex of dimension $n-1$.
In general $\dim \Delta_{k,E} = n - 1 $, 
as $\Delta_{k,E} \subset H(\mbfe_{E}, k)$.
Moreover, $\Delta_{n-k,E}$ is the image of $\Delta_{k,E}$ under the involution $(-)^* : \RR^E \to \RR^E$ given by $\mbfp \mapsto \mbfp^* = -\mbfp + 1$.
Thus, 
\begin{align} 
    \label{eq:CDindexUnderDuality}
    \Psi_{cd}({\Delta_{n-k,E}}) = \Psi_{cd}({\Delta_{k,E}}).
\end{align}
To apply Lemma~\ref{lm:CDIndexStratified} we describe $\calF(\Delta_{k,n})$.
This is a standard fact, see e.g.~\cite[Section 10.3]{jos21} for an exposition.
First, 
facets of the form $\Delta_{k,E} \cap H(\mbfe_i, 0)$ are called \emph{deletion facets}, 
and $\Delta_{k,E} \cap H(\mbfe_i, 1)$ are \emph{contraction facets}.
Intersections of facets determine faces as follows:

\begin{lemma} 
    \label{lm:FacesOfDeltakn}
    Let $n \in \ZZ_{\ge 1}$ and $k \in \aset{0, \dots, n}$.
    There is a bijection
    \begin{align} 
    \label{eq:FacesOfDelta}
     \Gamma :=  \left\{ \begin{array}{c}
     C, D \subset E \text{ s.t. } C \cap D = \varnothing, C \cup D \ne \varnothing\\
      \card{C} < k \text{ and } 
     \card{D} < n - k
  \end{array}\right\}
  \longleftrightarrow
       \left\{ \begin{array}{c}
        \text{ faces $\sigma$ of } \Delta_{k,n} \\
        \text{s.t. } \dim \sigma \ge 1
  \end{array}\right\},
    \end{align}
    where $(C,D) \in \Gamma$ is sent to 
    $\sigma_{C,D} = \Delta_{k,n} \cap 
    \bigcap_{i \in C} H(\mbfe_i, 1) 
    \cap  \bigcap_{j \in D} H(\mbfe_j, 0)$,
    and moreover this $\sigma_{C,D}$ is isomorphic to the hypersimplex 
    $\Delta_{k-\card{C},n- \card{C \cup D}}$.
\end{lemma}

\begin{remark} 
    Note that if $k=1$, there are only deletion facets, and if  $k = n-1$, there are only contraction facets.
\end{remark}

Plugging in Lemma~\ref{lm:FacesOfDeltakn} in Equation~\eqref{eq:StratifyCDIndex} gives us the following:

\begin{proposition} 
    \label{prop:PhiForHypersimplex}
    Let $n \in \ZZ_{\ge 1}$ and $k \in \aset{0, \dots, n}$. 
    We have that
    \begin{align*} 
        \Phi(\Delta_{k,n}) =& 
        \sum_{  \substack{(C,D) \in \Gamma \\  1\leq  \card{C\cup D} \leq n-2}}  
        \Psi_{cd}(\sigma_{C,D})b (a-b)^{\card{C \cup D} - 1} \\
        \nonumber &+ (a-b)^{n-1} + \binom n k b (a-b)^{n-2} \\
        =&
     \sum_{(i,j) \in \pi_{\#}(\Gamma)}\binom{n}{i}  \binom{n-i}{j} \Psi_{cd}(\Delta_{k-i, n-i-j}) b(a-b)^{i+j-1} \\
         &+ (a-b)^{n-1} + \binom{n}{k} b(a-b)^{n-2}. 
    \end{align*}
    Here $\pi_\# $ takes a pair $(C,D)$ to the cardinalities $(\card{C}, \card{D})$. 
\end{proposition}

As $\Delta_{k,n}$ is highly symmetric, 
we can calculate $\Psi_{ab}$ by partitioning into sums $\Psi_{ab}^{(i)}$ 
of contributions of chains $\sigma_1 \subset \dots \subset \sigma_s$ such that $\dim \sigma_s = i$;
that is~$\Psi_{ab} = \sum_{i = 0}^{k-1} \Psi_{ab}^{(i)}$.
We exemplify this for $\Psi_{ab}(\Delta_{2,5})$, 
and also perform some substitutions towards $\Psi_{cd}(\Delta_{2,5})$.

\begin{ex} 
    \label{ex:U25}
For the $ab$-index of $\Delta_{2,5}$ we sum over $-1\leq i \leq 3$, where $i=-1$ encodes the empty chain. 
For the substitutions we use Lemma~\ref{lm:AlgebraForABCD} 
and recursive knowledge of how to compute $\Psi_{cd}$ from $\Psi_{ab}$ for simplices on groundsets $E$ with fewer elements than $n$:
\begin{align*} 
    \Psi_{ab}^{(-1)} &= (a-b)^4 = (c^2-2d)^2; \\
   \Psi_{ab}^{(0)} &= \tbinom{5}{2}b(a-b)^3 = 10[d,c](a-b) + 10(a-b)^2b(a-b) \\
                   &= 10[d,c](c-2b) + 10(c^2-2d)(d-cb) \\ 
                   &= 10([d,c]c + c^2d - 2d^2) - 10(2[d,c] + c^3 - 2dc)b; \\ 
   \Psi_{ab}^{(1)} &= \tbinom{5}{1}\tbinom{4}{2}\Psi_{ab}(\Delta_{1,2})b(a-b)^2 = 30 \Psi_{cd}(\Delta_{1,2})b(a-b)^2 \\
                   &=  30 \Psi_{cd}(\Delta_{1,2})[d,c] + 30 \Psi_{cd}(\Delta_{1,2}) (c^2 - 2d)b; \\
   \Psi_{ab}^{(2)} &= \tbinom{5}{0}\tbinom{5}{2}\Psi_{ab}(\Delta_{2,3})b(a-b) + \tbinom{5}{1}\tbinom{4}{1}\Psi_{ab}(\Delta_{1,3})b(a-b) \\
    &= 30\Psi_{ab}(\Delta_{1,3})b(a-b) = 30\Psi_{cd}(\Delta_{1,3})d - 30\Psi_{cd}(\Delta_{1,3})cb;   \\
 \Psi_{ab}^{(3)} &=\tbinom{5}{0}\tbinom{5}{1}\Psi_{ab}(\Delta_{2,4})b + \tbinom{5}{1}\tbinom{4}{0}\Psi_{ab}(\Delta_{1,4})b\\
 &= 5\Psi_{cd}(\Delta_{2,4})b + 5\Psi_{cd}(\Delta_{1,4})b.
\end{align*}
As promised, our substitutions eliminate $a$ and bring the $b$ to the end of the monomials.
 
Summing the $\Psi_{ab}^{(i)}$ gives 
\begin{align}
\Psi_{ab}(\Delta_{2,5}) &= 51 a b a b + 29 a b a^2 + 31 a b^2 a + 9 a b^3 + 29 a^2 b a + 21 a^2 b^2 + 9 a^3 b + a^4  
 \\ & \hspace{0.5 cm}+ 51 b a b a + 29 b a b^2 + 31 b a^2 b + 9 b a^3 + 29 b^2 a b + 21 b^2 a^2 + 9 b^3 a + b^4. \nonumber 
\end{align}
Note that the expression is symmetric under exchanging $a$ with $b$, as argued in Lemma~\ref{lm:vanishingb}.
Summing all the expressions in terms of $b,c,d$ yields
    \begin{align} 
        \label{align:}
        \Psi_{cd}(\Delta_{2,5}) = p(c,d)+q(c,d)\cdot b.
    \end{align}

    By Lemma~\ref{lm:vanishingb} we ignore contributions of words that end in $b$ to speed up calculations:
\begin{equation*}
    \Psi_{cd}(\Delta_{2,5}) = 20cdc + 8c^2d + c^4 + 8dc^2 + 14d^2. \qedhere
\end{equation*}
\end{ex}   

Carrying out Example~\ref{ex:U25} in generality gives us a recursive formula: 
\begin{proposition}
Let $n \in \ZZ_{\ge 1}$ and $k \in \aset{2, \dots, n-2}$. 
Write $n_2$ for the remainder of $n$ modulo $2$.
We have that
    \begin{align*}
        \Psi_{cd}(\Delta_{k,n}) =& \sum_{i=0}^{k-1} \binom{n}{i} \sum_{j=\max(0,1-i)}^{n-k-1} \binom{n-i}{j} \Psi_{cd}(\Delta_{k-i,n-i-j}) g_{cd}(i+j-1) \\&+ \binom{n}{k} g_{cd}(n-2) + (c^2-2d)^{\lfloor{\frac{n-1}{2}}\rfloor}(c-2b)^{(n-1)_2}.
\end{align*}
\end{proposition}

\begin{remark} 
    There are other ways to approach this computation; 
    our down-to-earth  way generalizes well to connected split matroids.
    For other approaches see for example the unpublished manuscript \cite{gm}, 
    who use generating functions and the theory of Coxeter groups.
    This way fully exploits the fact that the vertices of $\Delta_{k,n}$ 
    lie in one single orbit under the action of the symmetric group $\operatorname{Sym}(E)$.
\end{remark}

\section{Matroids, cuts of the hypersimplex, and the case of one cut}
    \label{sec:Matroids}

\subsection{Matroids as slices of \texorpdfstring{$\Delta_{k,n}$}{Delta(k,n)}}
    \label{sub:SlicesOfDeltaKN}

A matroid base polytope $\scrP(M) \subset \RR^E$ is a 0/1-polytope 
with $\vertici{\scrP(M)} \subset \vertici{\Delta_{k,n}}$
and edges equal to parallel translates of $\mbfe_i - \mbfe_j$ for some $i,j$.
The hypersimplex $\Delta_{k,n}$ is our main example, 
and we recall how to get any other $\scrP(M)$ by chopping off pieces from $\Delta_{k,n}$.

The vertices $\vertici{\scrP(M)}$ are indicator vectors of the family of bases $\calB(M)$,
the integer $k$ is the \emph{rank} of the matroid $M$,
and the condition on the edges encodes the usual \emph{basis exchange property}.
The \emph{connected components} are the equivalence classes $E_1, \dots, E_c$ of $E$  under the relation generated by 
$i \sim j$ when there is an edge in $\scrP(M)$ that is a parallel translate of $\mbfe_i - \mbfe_j$.
If $c=1$ we say that $M$ is \emph{connected}.
We have the decomposition $M = ( M|{E_1}) \oplus \dots \oplus ( M|{E_c})$;
here $M|{E_i}$ is the restriction of $M$ to $E_i$.
The bases of the restriction are $B \cap E_i$ where $B$ is in $\operatorname{arg-max}_{B \in \calB(M)} \langle \mbfe_{E_i}, \mbfe_B \rangle$. 
This induces a decomposition of base polytopes
$\scrP(M) = P(M|{E_1}) \times \dots \times P(M|{E_c})$,
so in particular $\dim \scrP(M) = n - c$.
In light of the product formula in Equation~\eqref{eq:ProductFormula}, we focus on connected matroids.

Recall, e.g.~from \cite{oxl06a}, the following:
subsets of a basis are called \emph{independent};
inclusion-wise minimal dependent sets are \emph{circuits};
unions of circuits are \emph{cyclic sets};
the \emph{rank} $\rk S$ of $S \subset E$ is $\max_{B \in \calB(M)} \langle \mbfe_S, \mbfe_B \rangle$, 
and $F \subset E$ is a \emph{flat} if $\rk(F) < \rk(F \cup \aset{i})$ for all $i \in E \setminus F$.
The dual $M^*$ of $M$ is given by the image in $\Delta_{n-k,n}$ of $\scrP(M)$ under the involution $(-)^* : \mbfp \mapsto -\mbfp + \mbf 1$.
Flats and cyclic sets ordered by inclusion form two distinct lattices: $\calL(M)$ and $\Cyc(M)$.
Moreover, $\calL(M)$  is anti-isomorphic to $\Cyc(M^*)$ of the dual $M^*$:
the map $F \mapsto E \setminus F$ is an order-reversing bijection.
This follows from the fact that the complement of a \emph{cocircuit}, 
i.e.~a circuit of the dual matroid, 
is a \emph{hyperplane}, 
i.e.~a rank $k-1$ flat.
Cyclic flats $\calZ(M)$  are a sublattice of both $\calL(M)$ and $\Cyc(M)$.
The algebraic operations are compatible as follows:

\begin{proposition} 
    \label{prop:MeetAndJoinOfZM}
Let $M$ be a matroid and $F, G \in \calZ(M)$.
The join $F \join_{\calZ} G$ equals the join of $F, G$ in $\calL(M)$.
Likewise, the meet $F \meet_{\calZ} G$ equals the meet of $F, G$ in $\Cyc(M)$.
\end{proposition}

\begin{proof} 

For a proof we refer to \cite[Section 2]{bd08}. 
Specifically, 
the join in $\calZ(M)$ coincides with the closure of the union, 
$F \vee_{\calZ}G = \operatorname{cl}(F \cup G)$, 
and the meet is the union of all circuits contained in the intersection, 
$F \wedge_{\calZ} G =\bigcup \{C \in \calC(M) : C \subseteq F\cap G\}$.
\end{proof}

Any lattice appears as $\calZ(M)$ for some matroid $M$ \cite[Thm 2.1]{bd08}.
Moreover, the pair $\calZ(M)$ and $r|_{\calZ(M)}$ determines $M$:

\begin{theorem}[Theorem 2.2 in \cite{bd08}]
Let $M$ be a matroid, $\calZ(M)$ its cyclic flats, $\rk : 2^E \to \ZZ_{\ge 0}$ its rank function.
    We have that
\begin{align} 
    \label{eq:CuttingOut}
    \scrP(M) = \aset{x \in \Delta_{k,n} \suchthat \langle \mbfe_F, x \rangle \le \rk(F) \quad \forall \text{ proper } F \in \calZ(M) }.
\end{align} 
\end{theorem}

We denote by $ \properZ(M)$ the set of proper cyclic flats.
Geometrically, the pairs $(F, r) \in \properZ(M)$  
tell us how to carve out $\scrP(M)$ from $\Delta_{k,n}$ using few cuts.
We slightly abuse notation and write $H(F, r)$ for $H(\mbfe_F,r)$,
and even more, 
if $r$ is clear from the context, then we simply write $H_F$.
So Equation~\eqref{eq:CuttingOut} becomes
\[\scrP(M) = \Delta_{k,n} \cap \bigcap_{(F,r) \in \properZ(M)} H^-(F, r). \]
In the next subsection we look at a class of matroids for which these cuts are minimal.
Moreover, geometrically and combinatorially the interaction of the pairwise cuts is \emph{nice}. 

\begin{remark} 
    \label{remark:LooplessAndColoopless}
A singleton $\aset{e}$ is a \emph{loop} if it is dependent in $M$, 
and a \emph{coloop} if it is dependent in $M^*$;
i.e.~$e$ is contained in every basis of $M$.
If $S, T$ are the sets of loops and coloops, respectively, 
then $M = M | S \oplus M | (E \setminus (S \cup T)) \oplus M | T$.
Thus, assuming that $M$ is connected implies, in particular, that $M$ has no loops or coloops.
\end{remark}

\subsection{Connected split matroids}
    \label{sub:ElementarySplitMatroids}
For $A \subset E$ the \emph{cover} of $A$ is the family $\Cover A$ of $k$-subsets $B \subset E$ such that $\card{B \cap A} \ge \rk A + 1$.
These are the vertices $\mbfe_B \in \vertici{\Delta_{k,n}}$ excluded by a given $F$ in Equation~\eqref{eq:CuttingOut};
e.g. those in the interior $(H^+(F, r))^\circ$.

In \cite{fs24c} Ferroni and Schröter capture when $\calB(M) \cup \Cover A$ is a matroid,
which means relaxing the inequalities from Equation~\eqref{eq:CuttingOut}.

\begin{definition} 
  A subset $S \subset E$ is \emph{stressed} if both the restriction $M|A$ and the contraction $M/A$ are uniform matroids.
\end{definition}

\begin{proposition}[Theorem 3.12 in \cite{fs24c}]
    \label{proposition:MatroidUnionCover}
    Let $M$ be a matroid on $E$, and $A \subset E$ a subset.
    If $A$ is stressed, then $\calB(M) \cup \Cover A$ is the family of bases of a matroid on $E$.

\end{proposition}

Moreover, stressed subsets can be characterized in terms of the lattice of cyclic flats.
 
\begin{proposition}[Proposition~3.9 in \cite{fs24c}]
    \label{proposition:StressedIsSaturatedChain}
Let $M$ be a matroid on $E$ without loops or coloops.
A subset  $A \subset E$ is a stressed subset with $\Cover A \ne \varnothing$ 
if and only if  
$\varnothing \subset A \subset E$ is a saturated chain in $\calZ(M)$,
i.e.~$A$ is inclusion-wise both maximal and minimal among proper cyclic flats.
\end{proposition}

One can iteratively relax stressed subsets.
On $\calZ(M)$ each relaxation deletes a saturated chain of the form $\varnothing 
\lessdot A \lessdot E$  
and everything else remains untouched \cite[Proposition 3.19]{fs24c}.
We arrive thus to connected split matroids:

\begin{definition} 
    \label{definition:ElementarySplitMatroid}
    A connected matroid is split if relaxing all the stressed subsets with non-empty cover yields a uniform matroid. 
\end{definition}

The relation to the objects introduced in \cite[Def. 9]{js17} is elaborated in \cite[Sec.~4.2]{fs24c}.

\subsection{Cuspidal matroids}
    \label{sub:CuspidalMatroids}

For the calculation of the \texttt{cd}-index we analyze the pieces that are eliminated by cuts.

These are called \emph{cuspidal matroids} in \cite[Subsection~3.5]{fs24c}.

They depend on a choice of $F \subset E$ and an integer $r$ that determine a hyperplane $H(F,r)$ to perform a cut: 
    \begin{align} 
        \label{eq:cuspidal}
        \scrP(\Lambda^{r,F}_{k,n}) = \Delta_{k,n} \cap H^-(F,r).
    \end{align}
    The bases of $\Lambda^{r,F}_{k,n}$ are the $k$-subsets $B \in \tbinom{E}{k}$ satisfying $\card{B \cap F} \le r$. 
    If $r \ge \min(\card{F}, k)$ the cut is trivial, so we recover $\Delta_{k,n}$.
    Otherwise, for any $B' \in \tbinom F r $ and $e \in F^*$ the set  $B' \cup \{e\}$ is independent; 
    so $F$ is a flat.
    The circuits of $\Lambda^{r,F}_{k,n}$ are exactly the subsets $C \in \tbinom{F}{r+1}$, $C \in \tbinom{E}{k+1}$. 
    Hence $F$ is a cyclic flat, and it is the unique proper one.
    The isomorphism type of $ \Lambda^{r,F}_{k,n}$ depends solely on the parameters $k$ and $n$, which define the ambient hypersimplex, alongside $r = \rk(F)$ and  $h = \card{F}$,  
    which dictate the geometry of the cut.
    In the notation from \cite{fs24c} we have
\[
    \scrP(\Lambda_{k-r,k,n-h,n})= \aset{x \in \Delta_{k,n} \suchthat \sum_{i=1}^{h} x_i \leq r}.
\]
To understand the \texttt{cd}-index of these pieces $\Lambda$,
we refine Lemma~\ref{lm:FacesOfDeltakn} to describe the faces of $\Lambda$.

\begin{proposition} 
    \label{prop:FacesOfCuspidal}

Let $(F,r) \in \calP(E) \times \NN$ and denote the complement by $F^* = E \setminus F$.  
Set $h = \card{F}$, 
and write $r^* = k - r$,  $h^* = n - h$.
Choose a face $\sigma_{C,D} \prec \Delta_{k,n}$, given by subsets $C, D \subset E$:
\begin{enumerate} 
    \item We have $\sigma_{C,D} \subset H^+(F,r)$ if and only if at least one of the following holds:
   \begin{align} 
       \label{eq:ContainedInHPlus}
     \card{C\cap F}\geq r \quad \text{ or } \quad  \card{D\cap F^*}\geq h^*-r^*.
   \end{align}
\item \label{item:TheDual} Dually, $\sigma_{C,D} \subset H^-(F,r)$ if and only if at least one holds:
   \begin{align} 
       \label{eq:ContainedInHMinus}
     \card{C\cap F^*}\geq r^* \quad \text{ or } \quad  \card{D\cap F}\geq h-r.
   \end{align}

\item \label{item:SmallerCuspidal}  The polytope $\sigma_{C,D} \cap H^-(F,r)$ is isomorphic to $\Lambda^{r', F'}_{k',n'}$ 
  with parameters 
  $r' = r - \card{F \cap C}$, 
  $F' = F \setminus (C \cup D)$, 
  $k' = k - \card{C}$ and 
  $n' = n - \card{C \cup D}$.
\end{enumerate}
\end{proposition}

\begin{proof} 
Recall that $\sigma_{C,D}= \Delta_{\card{C},C}\times \Delta_{k',n'}\times \Delta_{0,D}$:
the vertices $\mbfe_B$ of $\sigma_{C,D}$ are those that satisfy $C \subset B$ and $D \cap B = \varnothing$.
For the first item: 
\begin{itemize} 
  \item  If $\card{C \cap F} \ge r$, then $\langle \mbfe_B, \mbfe_F \rangle = \card{B \cap F} \geq \card{C \cap F} \ge r$ for all vertices of $ \sigma_{C,D} $.
  \item Assume that $\card{D\cap F^*}\geq h^*-r^*$.  
    Note that $\angbra{\mbfe_B, \mbfe_F} = k -  \angbra{\mbfe_B, \mbfe_{F^*}} $, 
    and that $B \cap D = \varnothing$ implies that $\card{B \cap F^*}$ 
    is at most $\card{F^* \setminus D} = \card{F^*} - \card{F^* \cap D} \le h^* - (h^* - r^*) = r^*$.
    So $\angbra{\mbfe_B, \mbfe_F} \ge k - r^* = r$. 
  \item Conversely, assume that $\card{C \cap F} \le r-1$ and $\card{D\cap F^*} \le h^*-r^*-1$.
    The latter yields $\card{F^* \setminus D} 
        = \card{F^*} - \card{F^* \cap D} \ge  r^*+ 1$.
    We choose a basis $B = C  \cup S \cup T$,
    with 
    \begin{itemize} 
      \item $S \subset F \setminus (C \cap F)$ 
          and $\card{S} = (r-1) - \card{C \cap F} \ge 0$; 
      \item $T \subset F^* \setminus (D \cup C) $ 
        and $\card{T} = k - \card{S} - \card{C} = r^* + 1 - \card{C \cap F^*}$;
        this is possible because $C$ and $D$ are disjoint, implying that
        $\card{ F^* \setminus (D \cup C) } = \card{F^* \setminus D} - \card{C \cap F^*}$.
    \end{itemize} 

\end{itemize}
With these choices we get that $\card{B} = k$, 
and $\langle \mbfe_B, \mbfe_F \rangle = \card{(C \cup S) \cap F} = r - 1$, showing that $\mbfe_B$ is not contained in $H^+(F,r)$.

The second item proceeds dually. 
For the third item, let $M$ be the matroid with polytope $\scrP(M)$ equal to $\sigma_{C,D}$.
All the bases of $M$ contain $C$ and avoid $D$, 
so we take the matroid contraction $M/C$ and then the deletion $(M/C)\setminus D$.
This is isomorphic to $\Delta_{k',n'}$.
The set $F$ gets mapped to $F \setminus (C \cup D)$ and its rank in $(M/C)\setminus D$ equals $r - \card{F \cap C}$, as claimed.
\end{proof} 

\begin{remark} 
    \label{re:AFaceIsNotCut}
    Assume that $\sigma_{C,D} \subset H^-(F,r)$.
    Consider the inequalities from Equation~\eqref{eq:ContainedInHMinus}:
    if $\card{C\cap F^*}\geq r^*$, then $r' \ge k'$;
    otherwise $\card{D\cap F}\geq h-r$, yielding $r' \ge \card{F'}$.
    Both conditions imply $\Lambda^{r', F'}_{k', n'} = \Delta_{k', n'}$, 
    as expected since $\sigma_{C,D} \subset H^-(F,r)$.
\end{remark}

\begin{lemma} 
    \label{lm:PhiForCuspidal}

    Let $F \subset E$ and choose $r \in \ZZ_{\ge 1}$. 
    Write $\EmVe$ as in Equation~\eqref{eq:EmptyVertices}, $\Gamma$ as in Equation~\eqref{eq:FacesOfDelta}, 
    and $\tau = \Delta_{k,n} \cap H(F, r)$.
    We have that:
    \begin{align} 
        \label{eq:LambdaCd}
        \Phi(\Lambda^{r,F}_{k,n}) &= 
        \sum_{ \substack{
                (C,D) \in \Gamma \\ 
                \card{C \cap F}< r \\
                \card{D \cap F^*} < h^* - r^* 
             } }  
        \Psi_{ab}\left(\sigma_{C,D} \cap H(F, r)^-
        \right)b (a-b)^{\card{C \cup D} - 1} \\
        \nonumber &+ \sum_{ \substack{\eta \prec \tau \\ \dim(\eta)\geq 1}} \Psi_{ab}(\eta) b (a-b)^{\codim \eta} + \EmVe(\Lambda^{r,F}_{k,n}).
    \end{align}
\end{lemma}
\begin{proof} 
    From the count in Proposition~\ref{prop:PhiForHypersimplex},
    use Equation~\eqref{eq:ContainedInHPlus} to throw away all the $\sigma_{C,D}$ contained in $H^+(F,r)$.
    This gives the first sum in the right hand side.
    The faces contained in $\tau$ need to be readded, giving the second sum.
\end{proof}

The faces $\eta \prec \tau$ from Equation~\eqref{eq:LambdaCd} can be understood via the following result.

\begin{lemma} 
    \label{lm:tauincuspidal}
    The polytope $\tau = \Delta_{k,n} \cap H(F, r)$  is isomorphic to $\Delta_{r,h} \times \Delta_{k - r, n - h}$.
\end{lemma}

\begin{proof} 
  Recall that $h = \card{F}$ and $n - h = \card{F^*}$. 
    The bases $B$ for which $\mbfe_B$ is in $H_F = \aset{\angbra{\mbfe_F, \mbfe_B} = r}$ are precisely those for which $\card{F \cap B} = r$ 
    and $\card{F^* \cap B} = k - r$.
\end{proof}

  \begin{remark}
\label{rmk:FacesLambda}
Formula \eqref{eq:LambdaCd} is recursive because by Item~\ref{item:SmallerCuspidal} of Proposition~\ref{prop:FacesOfCuspidal} and Remark~\ref{re:AFaceIsNotCut} each $\sigma_{C,D} \cap H^-(F,r)$ is itself cuspidal or a hypersimplex.
We do not attempt to write an appealing formula, 
however we provide the function \texttt{compute\textunderscore P\textunderscore Lambda\textunderscore cd(k, n, r, h)} in our computer program.
\end{remark}

Using the same ideas as from Example~\ref{ex:U25} we get a recursive function to calculate the \texttt{cd}-index directly without passing through the \texttt{ab}-index.

\section{The geometry of connected split matroids}
    \label{sec:TheGeometryOfSplitConnectedMatroids}

    We explore several geometric features of the base polytope for connected split matroids, 
aimed at facilitating the application of Kim's formula.
    Throughout this section, we keep the following illustrative example in mind.

\begin{figure}[htpb]
%
    \centering
    \includegraphics[width=0.5\linewidth]{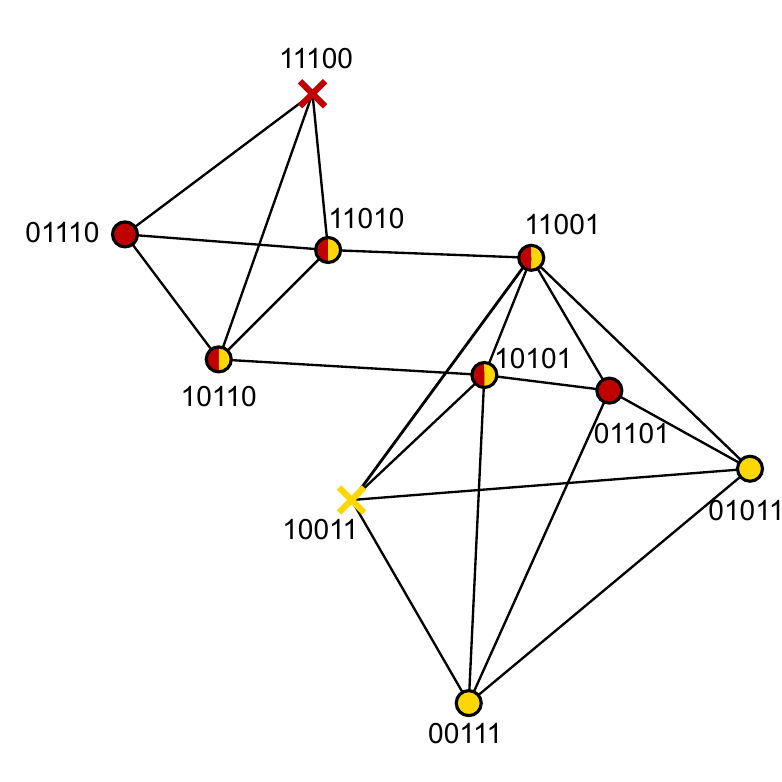}
    \caption{
      The cuts $H_F = H(123,2)$ and $H_G = H(145,2)$ delete, respectively,        
      the vertices 11100 and 10011 from $\Delta_{3,5}$. 
      While not all edges are shown, the figure displays a pair of deletion and contraction facets (a simplex and a bipyramid), 
    alongside the intersection $\Delta_{3,5} \cap H_F \cap H_G$  
    which corresponds to a modular pair of flats (the central square $\tau$).
The face $\sigma \in \Delta_{k,n}$ satisfying $\tau = \sigma \cap H_G$ is the bipyramid formed by the deleted vertices. 
These lie in the regions $-+$ and $+-$, illustrating Corollary~\ref{cor:PlusPlusIsEmpty}. }
    \label{fig:Delta35}
\end{figure}

\begin{example} 
    \label{ex:ModularAffectsCalculation}
  Consider $\scrP(M) \subset \Delta_{3,5}$ with non-bases $F = \aset{1,2,3}$ and $G = \aset{1,4,5}$, 
  i.e.~obtained via the cuts $H^-(F,2)$ and $H^-(G,2)$. 
Figure~\ref{fig:Delta35} highlights the vertices adjacent to $\mbfe_{\{1,2,3\}}$ in red, 
and those adjacent to $\mbfe_{\{1,4,5\}}$ in yellow; 
these belong to the faces of $\scrP(M)$ contained in $H_F$ and $ H_G$, respectively. 
We make some observations:
\begin{itemize} 
    \item Bicolored vertices: The set of bicolored vertices forms a matroid isomorphic to:
      \begin{align*}
        \Delta_{1,F \cap G} \times \Delta_{1,F \setminus G} \times \Delta_{1, G \setminus F} \times \Delta_{0, E \setminus (F \cup G)} &= \\ 
 \Delta_{1,\aset{1}} \times \Delta_{1, \aset{2,3}} \times \Delta_{1, \aset{4,5}} \times \Delta_{0, \varnothing} &\isom \Delta_{1,2} \times \Delta_{1,2}. 
\end{align*}

\item Modular pairing: The two cyclic flats $F$ and $G$ form a modular pair: 
  \[ \rk(F) + \rk(G) = 2+2=1+3=\card{F\cap G}+\rk(F\cup G).\]
\item Mutually exclusive cuts: No vertex is eliminated twice; 
  using notation introduced below, $\Delta_{3,5} \cap O(++)$ is empty.
\end{itemize}
\end{example}

\subsection{Cuts only intersect for modular pairs}
    \label{sub:CutsOnlyIntersectForModularPairs}

The first two observations of Example~\ref{ex:ModularAffectsCalculation} follow from the work of Ferroni and Schröter \cite{fs25}. 
We provide a streamlined proof here.

\begin{lemma} 
    \label{lemma:TwoCuts}
    Let $M$ be a connected split matroid, and let $(F,r), (G,s)$ be distinct elements of~$\properZ(M)$. 
    Set $Q = \Delta_{k,n} \cap H^+(F,r) \cap H^+(G,s)$ and define 
    $a = \card{F \setminus G}$, 
    $b = \card{G \setminus F}$ 
    and $\gamma = \card{F \cap G}$.
    The flats $F$ and $G$ form a modular pair
    if and only if
    $Q$ is non-empty.
    Furthermore, if $Q \ne \varnothing$, then
    \begin{align}
        Q &= \Delta_{k,n} \cap H(F,r) \cap H(G,s) \\
          &= \Delta_{\gamma, F \cap G} 
            \times \Delta_{r - \gamma, F \setminus G} 
            \times \Delta_{s - \gamma, G \setminus F} 
            \times \Delta_{0, E \setminus (F \cup G)} \\
\nonumber          & \isom \Delta_{r - \gamma, a} \times \Delta_{s - \gamma, b}.
     \end{align} 
\end{lemma}

\begin{proof} 
  The polytope $Q$ is non-empty if and only if there exists a vertex $\mbfe_B\in \vertici{\Delta_{k,n}}$ satisfying $\langle \mbfe_F, \mbfe_B \rangle \geq r$ and $\langle \mbfe_G, \mbfe_B \rangle\geq s$.
Suppose such a vertex $\mbfe_B$ exists.
We have the following sequence of inequalities:
\begin{align} 
    \label{eq:CutsAreEmpty}
    k + \card{F \cap G} &\geq \langle \mbfe_{F\cup G}, \mbfe_B \rangle + \langle \mbfe_{F\cap G}, \mbfe_B \rangle \nonumber \\ 
                        &= \langle \mbfe_{F}, \mbfe_B \rangle + \langle \mbfe_{G}, \mbfe_B \rangle \nonumber \\  
                        &\geq r + s \nonumber \\
                        &\geq \rk(F\cup G) + \rk(F\cap G).
\end{align}
Here, the first inequality follows from $\angbra{\mbfe_S, \mbfe_B } 
\le \min(\card{S}, \card{B})$;
the second from our hypothesis on~$B$;
and the third is submodularity.
Since $M$ is connected and split, 
and $F,G \in \properZ(M)$, we have that $E = F \vee_{\calZ}G = \operatorname{cl}(F \cup G)$,
yielding~$\rk(F \cup G) = k$.
Likewise,  $\varnothing = F \wedge_{\calZ} G =\bigcup \{C \in \calC(M) : C \subseteq F\cap G\} $ implies that $F \cap G$ is independent, so $\rk(F \cap G) 
= \card{F \cap G}$. 
This forces equality throughout Equation~\eqref{eq:CutsAreEmpty}, which implies that $F$ and $G$ are a modular pair, and moreover 
\[
\langle \mbfe_F, \mbfe_B \rangle = r, \quad \langle \mbfe_G, \mbfe_B \rangle = s, \quad \langle \mbfe_{F\cup G}, \mbfe_B \rangle=k, \quad \langle \mbfe_{F\cap G}, \mbfe_B \rangle=\card{F \cap G}.
\]
These equalities imply that $F\cap G\subseteq B \subseteq F\cup G$, and that  $\langle \mbfe_{F\setminus G}, \mbfe_B \rangle=  \langle \mbfe_{F}, \mbfe_B \rangle- \langle \mbfe_{F\cap G}, \mbfe_B \rangle=r-\card{F\cap G}$. The same holds symmetrically for $G$. 

Conversely, assume that $F$ and $G$ form a modular pair.
Any choice of  subsets 
$S_F \in \binom{F \setminus G}{r -\gamma}$ 
and $S_G \in \binom{G \setminus F}{s -\gamma}$ 
yields a basis $B = (F \cap G) \cup S_F \cup S_G$ 
such that $\mbfe_B$ satisfies Equation~\eqref{eq:CutsAreEmpty}. 
This yields $Q=\Delta_{k,n} \cap H(F,r) \cap H(G,s)\isom \Delta_{r - \gamma, a} \times \Delta_{s - \gamma, b}$
and also that $Q$ is non-empty.
\end{proof}

\subsection{No vertex is eliminated twice}
    \label{sub:NoVertexIsEliminatedTwice}
Recall that given $\aset{H = \ker (f : \RR^E \to \RR)}_{H \in \calA}$  a hyperplane arrangement in $\RR^E$, 
the sign vector $\sgn p$ of a point $p \in \RR^E$ is a vector in $\aset{-, 0, +}^\calA$ whose $H$-th entry $\sgn_H(p)$ records the sign of $f(p)$.
We denote by $O(\sgn p)$ the region of $\RR^E$ of points $x$ whose sign is $\sgn p$.
See Figure~\ref{fig:CryptoLemma}.

\begin{figure}[htpb]
    \centering
    \includegraphics[width=0.5\linewidth]{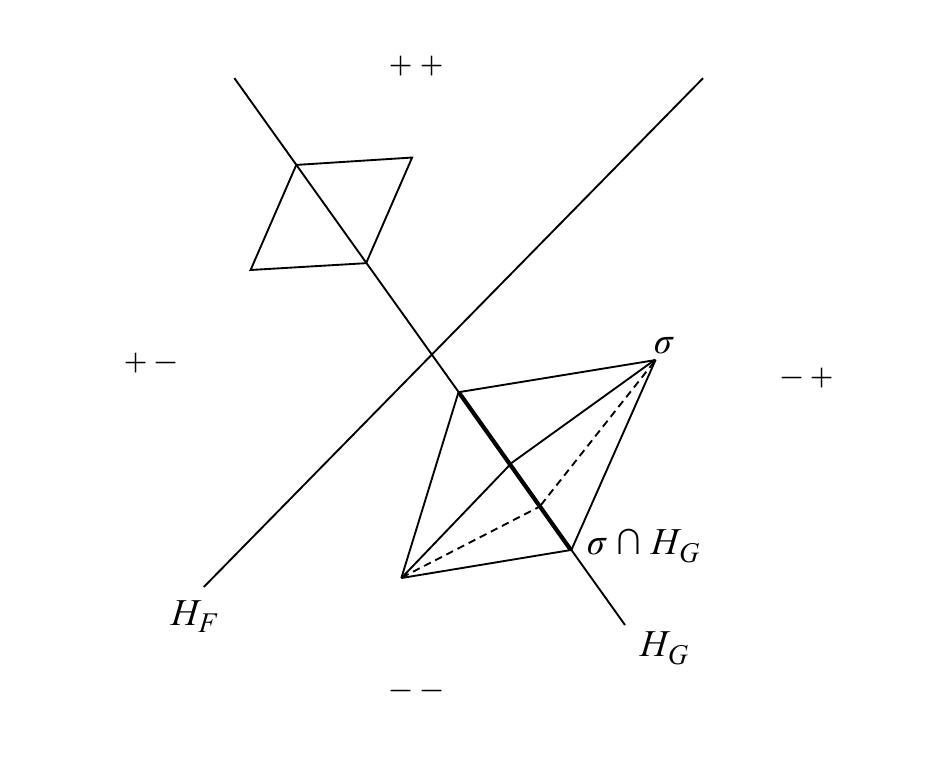}
    \caption{
    Two hyperplanes $H(F, r)$ and $H(F, s)$, defining regions  $O(++)$, $O(-+)$, $O(--)$, and $O(+-)$.
    A face of  $\Delta_{k,n}$ like the bipyramid $\sigma$ is possible, whereas a face like the square is not, because a piece lies in $O(++)$.
    }
    \label{fig:CryptoLemma}
\end{figure}

\begin{corollary}
    \label{cor:PlusPlusIsEmpty}
    Let $M$ be a connected split matroid such that $\properZ(M) = \aset{(F,r), (G,s)}$.
    The hypersimplex $\Delta_{k,n}$ is disjoint from $O(0+)$, $O(+0)$, and $O(++)$.
\end{corollary}

\begin{proof} 
If $Q = \Delta_{k,n} \cap H^+(F,r) \cap H^+(G,s)$ is empty we are done.
Otherwise Lemma~\ref{lemma:TwoCuts} shows that $Q = \Delta_{k,n} \cap H(F,r) \cap H(G,s)$, which implies that $\Delta_{k,n} \cap O(++)$ is empty, as desired. 
\end{proof}

By \cite[Theorem~3.2]{bd08} the condition that $\properZ(M) = \aset{(F,r), (G,s)}$ is a split matroid is equivalent to asking that the pair $(F,r)$ and $(G,s)$ satisfies:
\begin{itemize} 
  \item $r < \min(\card F, k)$, symmetrically for $G$.
  \item both $r$ and $\card{F^*} - (k - r)$ are positive; these are the rank of $F$ and corank of $F^*$. Symmetrically for $G$. 
  \item $r + s \ge k + \card{F \cap G}$.
\end{itemize}
The choices that yield a disconnected matroid are pairs of the form $(F, r)$ and $(F^*, k - r)$.

The previous corollary generalizes to an arbitrary number of cuts.

\begin{corollary}
    \label{cor:NoTwoPluses}
    Let $M$ be a connected split matroid, 
  with  $\calA = \aset{H_F}_{F \in \properZ(M)}$ the arrangement of cuts.
  For every $p \in \Delta_{k,n}$, if $\sgn p$ contains a $+$, then all the other entries are $-$. $\qed$
\end{corollary}
 
\subsection{Triple and double intersections}
    \label{sub:TripleAndDoubleIntersections}
    Fix a connected split matroid $\properZ(M)$ with cyclic flats $(F_i, r_i)$ and $(G,s)$
    and a face $\sigma \in \calT_{H(G,s)}(\Delta_{k,n})$.
    Using the previous results on sign vectors we 
    show that $\tau = \sigma \cap H(G,s)$ can appear in a double intersection,
    but never in a triple one.
    This technical result is the reason why Theorem~\ref{thm:AlmostValuativeInvariantsIntro} only considers modular pairs, equivalently flats $F$ and $G$ such that $H_F \cap H_G \cap \Delta_{k,n}$ is non-empty, 
    and never flats  such that $H_{F_1}  \cap H_{F_2} \cap H_G \cap \Delta_{k,n}$ is non-empty.
    When no confusion arises, we write $\calT_G$ for $\calT_{H(G,s)}$.

\begin{lemma} 
    \label{lm:SignFlipped}
    Let $\calA$ be an arrangement, 
    $\sigma \subset \RR^E$ a polytope,  
    $p$ a point in  $\sigma$,
    and $x$ a point in the relative interior $\relint \sigma$.
    There is a $q \in \relint \sigma$ such that for every $H \in \calA$ we have:
\[
\sgn_H(q) = 
  \begin{cases}
    -\sgn_H(p) \text{ if } \sgn_H(x) = 0,\\
\sgn_H(x) \text{ otherwise.}
  \end{cases}
\]
  \end{lemma}

\begin{proof} 
Extend slightly the segment from $p$ to $x$ to obtain an endpoint $q$.
It satisfies the desired properties because  $\sigma$ is convex and the elements of $\calA$ are hyperplanes. 

\end{proof}

In some applications of Lemma~\ref{lm:SignFlipped} we argue that there is no hyperplane of $\calA$ cutting the segment $[p,x]$ in the interior between $p$ and $x$. 
This extra condition implies that $\sgn x$ and $\sgn p$ agree in all entries except in $\sgn_H$.

\begin{lemma}
    \label{lm:MinusMinusEmpty}

    Let $M$ be a connected split matroid with $\properZ(M) = \aset{(F,r), (G,s)}$
    and choose $\sigma \in \calT_{G}(\Delta_{k,n})$.
    The intersection $\sigma \cap O(--)$ is empty 
    if and only if $\sigma \cap H_G \subset H_F$.
\end{lemma}

\begin{proof} 
Assume that $\sigma \cap O(--)$ is empty. 
Since $\sigma$ is in $\calT_{G}(\Delta_{k,n})$,
we can choose $x$ and $p$ in $\relint \sigma$
such that $\sgn_G(x) = 0$ and $\sgn_G(p) = -$.
By Lemma~\ref{lm:SignFlipped} we get a $p' \in \relint \sigma$ with $\sgn_G(p') = +$, and:
\begin{itemize} 
    \item if $\sgn_F(x) = -$, then $\sgn(p) = --$, contradicting that $\sigma \cap O(--)$ is empty. 
    \item  if $\sgn_F(x) = +$, then $\sgn(p') = ++$, contradicting Corollary~\ref{cor:PlusPlusIsEmpty}.
\end{itemize}
Thus, $\sgn_F(x) = 0 $, which means that $\sigma \cap H_G \subset H_F$ as desired.

Conversely, assume that $\sigma \cap H_G \subset H_F$. 
Choose $x$ in $\relint \sigma \cap H_G$, so $\sgn(x) = 00$.
If there were a $p \in \Delta_{k,n} \cap O(--) $,
Lemma~\ref{lm:SignFlipped} would  yield $p'$ in $O(++)$, contradicting Corollary~\ref{cor:PlusPlusIsEmpty}.
\end{proof} 

\begin{proposition} 
    \label{prop:NoTripleIntersection}
    Let $M$ be a connected split matroid such that $\properZ(M) = \aset{(F_1,r_1), \dots, (F_m, r_m), (G,s)}$.
    For all $\sigma \in \calT_G(\Delta_{k,n})$ the polytope $\tau = \sigma \cap H_G$ is disjoint from triple intersections $H_{F_1} \cap H_{F_2} \cap H_G$.
\end{proposition}

\begin{proof} 
Assume that $\tau = \sigma \cap H_G$ intersects a triple intersection $H_{F_1} \cap H_{F_2} \cap H_G$.
We consider the signature with respect to $F_1$, $F_2$, and $G$.
For all $x \in \sigma \cap H_G \cap H_{F_1} \cap H_{F_2}$ we have $\sgn(x) = 000$.
Since $\sigma$ is in $\calT_G(\Delta_{k,n})$, 
there is a $p \in \sigma$ with $\sgn_G(p) = +$.
Corollary~\ref{cor:NoTwoPluses} implies that $\sgn(p) = --+$.
Since $\sgn(x) = 000$, Lemma~\ref{lm:SignFlipped} yields $p'$ with $\sgn(p') = ++-$, contradicting Corollary~\ref{cor:NoTwoPluses}. 
\end{proof}
 
\begin{lemma}
    \label{lm:FacesInTheError}

    Let $M$ be a connected split matroid such that $\properZ(M) = \aset{(F_1,r_1), \dots, (F_m, r_m), (G,s)}$ 
    and $\sigma$ in~$\calT_{G}(\Delta_{k,n})$.

    There is no $p \in \sigma$ with all negative signature $\sgn(p) = -\dots-$
    if and only if $\sigma \cap H_G \subset H_F$ for exactly one choice of $F = F_i$.
\end{lemma}

\begin{proof} 
For the $(\Rightarrow)$ direction, suppose that there is no $p \in \sigma$ with an all-negative signature. 
Let $x$ be in $\relint \sigma \cap H_G$.
As $\sgn_G(x) =0$, Corollary~\ref{cor:NoTwoPluses} implies that $\sgn_{F_i}(x)$ is either $-$ or $0$ for every~$i \in [m]$;
and Proposition~\ref{prop:NoTripleIntersection} implies that $\sgn_{F_i}(x) = 0$ for at most one $i \in [m]$.
So we have two cases.
In the first case, 
there is a choice of $x \in \relint \sigma \cap H_G$ such that $\sgn_{F_i}(x) = -$ for every $i$, 
and then a straightforward application of Lemma~\ref{lm:SignFlipped} gives a point $p$ with an all-negative signature, contradicting our supposition.
In the second case, for every choice of $x \in \relint \sigma \cap H_G$ we have $\sgn_{F_i}(x) = 0$ for exactly one $i \in [m]$, 
and clearly we are done if we show that $i$ is independent of $x$.
Indeed, suppose that $x, x' \in \sigma \cap H_G$ yield distinct $i, i'$.
As the sign of $x$ is 0 for the $i$-th coordinate, it is $-$ for the $i'$-th one.
Likewise, in a reverted order, for $x'$. 
Note that then the signature of $(x+x')/2$ is $-$ for every $F_i$, yielding again a contradiction.

For the $(\Leftarrow)$ direction, without loss of generality suppose that $\sigma \cap H_G \subset H_{F_m}$.
Choose $x \in \sigma \cap H_G$, so $\sgn(x) = --\dots--00$.
If there were a point $p$ with $\sgn(p) = - \dots ---$, 
then Lemma~\ref{lm:SignFlipped} would give us a point $q$ whose last two entries of $\sgn(q)$ are positive, contradicting Corollary~\ref{cor:NoTwoPluses}.
\end{proof}

\subsection{Cuts do not interfere with each other}
    \label{sub:CutsDoNotInterfere}
The face cut out by $H_G$ is the same for all the family of relaxations.

\begin{lemma} 
    \label{lm:CutsDoNotInterfere}
  Let $M$ be a connected split matroid such that $(G, s) \in \properZ(M)$.
  For any relaxation $N$ of $M$  we have that
  \begin{align*}
      \scrP(M) \cap H_G &= \scrP(N) \cap H_G  
  \end{align*}

Moreover,  $\scrP(M) \cap H_G^+ = \scrP(N) \cap H_G^+$ if $G$ is not relaxed in $N$.
\end{lemma}

\begin{proof} 
Clearly the relaxation of $G$ does not change the intersection with $H_G$ since the discarded points are contained in $(H_G^+)^\circ$. Hence we can assume that $G$ is not relaxed. The result is a direct consequence of Corollary~\ref{cor:PlusPlusIsEmpty}. Namely, if one of the above equalities failed, then in the symmetric difference one would find a point with signature 0+, +0 or ++ for two specific cyclic flats.
\end{proof}



\subsection{\texorpdfstring{From $\calT_G$ to $\calT_G \cap H_G$}{From T\_(G,r) to T\_(G,r) cap H\_G}}
    \label{sub:FromTtoTH}
We now develop some tools to analyze the polytopes contributing to the sum of the correction term from Equation~\eqref{eq:KimsFormula}.
If $\calT$ is a collection of polytopes and $H$ a hyperplane, we denote by $\calT \cap H$ the set of intersections $\aset{\sigma \cap H \suchthat \sigma \in \calT}$.

\begin{lemma} 
    \label{lm:Bijection}
    Let $P$ be a polytope and $H$ a hyperplane.
    The map $\calT_H(P) \to \calT_H(P) \cap H$ 
    given by $\sigma \mapsto \sigma \cap H$
    is a bijection.  
\end{lemma}

\begin{proof} 
 Recall that $P$ is the disjoint union of its open faces and vertices, i.e.~$P = \vertici P \cup \bigcup_{\sigma \prec P} \relint \sigma$. 
     Since $H_G$ contains an interior point of every face $\sigma \in \calT_G(\Delta_{k,n})$, the intersection preserves the facial structure.

\end{proof}

The following result is used in the computation of the \texttt{cd}-index in Lemma~\ref{lm:ModularPairStatistics}. 

\begin{lemma}
    \label{lm:SplittingFace}
  The faces of $\Lambda^F_M$ that split a face of $\Delta_{k,n}$ 
are the faces of $H_F\cap \Delta_{k,n}\isom\Delta_{\rk(F),\card{F}} \times \Delta_{k-\rk(F),n-\card{F}} $ that contain a product of edges, 
i.e.~are not of the form $\Delta_{\rk(F) ,\card{F}} \times \{v\}$ or $\{w\} \times \Delta_{k-\rk(F) ,n-\card{F}}$ for any vertices $v$ of $\Delta_{k-\rk(F),n-\card{F}}$ and $w$ of $\Delta_{\rk(F),\card{F}}$.
\end{lemma}

\begin{proof} 
See the beginning of the proof of Proposition 2.8 of \cite{fs25}.
\end{proof} 

\section{Formulas for almost-valuative invariants}
    \label{sec:FormulasForAlmostValuativeInvariants}

\subsection{The \texttt{cd}-index}
    \label{sub:TheCDIndex}
Now we give our formula for two or more cuts. 
We slightly abuse notation and write $\Psi_{cd}(M)$ for $\Psi_{cd}(\scrP(M))$.
Also, we write $\Lambda^F_M$ and $\Lambda^M_F$ for the matroids 
with base polytopes $\Delta_{k(M),n(M)} \cap H^-(F, \rk(F))$ and $\Delta_{k(M),n(M)} \cap H^+(F, \rk(F))$, respectively.

\begin{theorem}
    \label{thm:CDIndexSplitMatroid}
    Let $M$ be a connected split matroid such that $\properZ(M) = \aset{(F_i,r_i)}_{i \in [m]}$. 
    We have that

    \begin{align} 
\label{eq:CDIndexSplitMatroid}
    \Psi_{cd}(M) = \Psi_{cd}(\Delta_{k,n}) + \sum_{F \in \properZ(M)} (\Psi_{cd}(\Lambda^F_M) - \Psi_{cd}(\Delta_{k,n}) )- \sum_{(F,G) \in \calM \calP(M)} W_{F,G}.
\end{align}
Here $\calM \calP(M)$ is the set of modular pairs $(F_i, F_j)$ of cyclic flats with $i < j$, and
\begin{align}
\label{eq:error}
    W_{F,G}=\sum_{\sigma \in \calS(F,G)} \Psi_{cd}(\sigma \cap H) \cdot d \cdot \Psi_{cd}(\Delta_{s, G} \times \Delta_{k-s, G^*}/(\sigma \cap H)).
\end{align}
Here $\calS(F,G)$ are faces $\sigma$ of $\calT_{G}(\Delta_{k,n})$ such that $\sigma \cap H_G \subset H_F$, and $s$ is the rank of $G$.
\end{theorem}

\begin{proof}
We proceed by induction on $m$.
For $m=0$ and $m=1$ the formula 
yields $\Psi_{cd}(\Delta_{k,n}) = \Psi_{cd}(\Delta_{k,n})$
and $\Psi_{cd}(\Lambda^F_M) = \Psi_{cd}(\Lambda^F_M)$,
respectively, which are trivially true.
We take $m=1$ as base case and assume that the formula is true for some $m$.
For $m + 1$, we write $(G,s)$ instead of $(F_{m+1}, r_{m+1})$.
Relaxing $(G,s) \in \properZ(M)$ yields $\properZ(N) = \aset{(F_1, r_1), \dots, (F_m, r_m)}$, a connected split matroid for which the inductive hypothesis applies.

We first deal with the valuative part.
We cut by $H_G$ and  apply Kim's formula to the following two polytopes:
$Q_1 = \scrP(N)$ and $Q_2 = \Delta_{k,n}$.
Lemma~\ref{lm:CutsDoNotInterfere} implies that $\scrP(N)\cap H_G = \Delta_{k,n} \cap H_G$; 
we denote this by $\tau_G$. 
Recall that $\tau_G$ is isomorphic to $\Delta_{s, G} \times \Delta_{k-s, G^*}$, 
and that $\scrP(N) \cap  H^+(G,s) = \scrP(\Lambda^M_G)$ by Lemma~\ref{lm:CutsDoNotInterfere} and Corollary~\ref{cor:NoTwoPluses}.
In the next few lines, a hat signifies intersection with $H_G$, i.e.~$\hat{\sigma} = \sigma \cap H_G$.
For $Q_1$ we have: 
\begin{align*}
  Q_1^- &= \scrP(N)\cap H^-(G,s) = \scrP(M),\\
  Q_1^+ &= \scrP(N) \cap  H^+(G,s) = \scrP(\Lambda^M_G),\\
  \hat Q_1 &= \scrP(N)\cap H_G = \tau_G, \text{ and }\\
  \Psi_{cd}(M) &= 
\Psi_{cd}(N) - \Psi_{cd}(\Lambda^M_G) + \Psi_{cd}(\tau_G ) \cdot c 
+ \sum_{\sigma \in \calT_G(N)} \Psi_{cd}(\hat{\sigma}) \cdot d \cdot \Psi_{cd}(\tau_G/\hat{\sigma}).
\end{align*}
Similarly, for $Q_2$ we have:
\begin{align*}
  Q_2^- &= \Delta_{k,n} \cap H^-(G,s)  = \scrP(\Lambda^G_M),\\
  Q_2^+ &= \Delta_{k,n} \cap  H^+(G,s)  =\scrP(\Lambda^M_G),\\
  \hat{Q}_2 &= \Delta_{k,n} \cap H_G = \tau_G, \text{ and }\\
  \Psi_{cd}(\Delta_{k,n}) &= \Psi_{cd}(\Lambda^G_M) + \Psi_{cd}(\Lambda_G^M) - \Psi_{cd}(\tau_G) \cdot c - \sum_{\sigma \in \calT_G(\Delta_{k,n})} \Psi_{cd}(\hat{\sigma}) \cdot d \cdot \Psi_{cd}(\tau_G/\hat{\sigma}).
\end{align*}
Adding both expressions and writing 
$\calE(\tau) = \Psi_{cd}(\tau) \cdot d \cdot \Psi_{cd}(\tau_G / \tau)$ for the correction:
\begin{align*}
\Psi_{cd}(M) - \Psi_{cd}(N) - \Psi_{cd}(\Lambda^G_M) + \Psi_{cd}(\Delta_{k,n}) =  \sum_{\sigma' \in \calT_G(N)} \calE(\sigma' \cap H_G) - \sum_{\sigma \in \calT_G(\Delta_{k,n})} \calE(\sigma \cap H_G).
\end{align*}

It remains to deal with the error term, the right-hand-side term above.
Lemma~\ref{lm:Bijection} implies it is equal to
\[
  -\left(\sum_{\tau \in \calT_G(\Delta_{k,n}) \cap H_G} \calE(\tau) - \sum_{\tau' \in \calT_G(N) \cap H_G} \calE(\tau') \right).
\]
We now make two claims. Our first claim is that $\calT_G(N) \cap H_G \subset \calT_G(\Delta_{k,n}) \cap H_G$, 
reducing the error to the sum of $\calE(\tau)$ for 
$\tau \in  (\calT_G(\Delta_{k,n}) \cap H_G) \setminus (\calT_G(N) \cap H_G)$.
The second claim is that the latter set is partitioned by the $F_i$ such that $(F_i, G)$ is a modular pair, that is
\[
\tau \in  (\calT_G(\Delta_{k,n}) \cap H_G) \setminus (\calT_G(N) \cap H_G) =
  \bigsqcup_{F_i} \aset{\sigma \in \calT_G(\Delta_{k,n}) \suchthat \sigma \cap H_G \subset H_{F_i}}.
\]
Those two claims and the induction hypothesis on $\Psi_{cd}(N)$ imply the result.

First claim: let $\tau'$ be in $\calT_G(N) \cap H_G$.
By Lemma~\ref{lm:Bijection} there is a unique $\sigma^N \in \calT_G(N)$ such that $\tau = \sigma^N \cap H_G$.
Recall that this means that $\sigma^N$ is a face of $\Delta_{k,n} \cap \bigcap_{i \in [m] }H^-(F_i, r_i)$.
Since $\sigma^N$ is in $\calT_G(N)$, 
there is $p \in \relint \sigma^N$ such that $\sgn_G(p) = +$.
By Corollary~\ref{cor:NoTwoPluses}, we have $\sgn_{F_i}(p) = -$.
Thus, $p$ is contained in $\bigcap_{i \in [m]} H^-(F_i, r_i)^\circ$, 
and since $p$ is in $\relint \sigma^N$, 
this implies that there is $\sigma \prec \Delta_{k,n}$ such that $\sigma^N = \sigma \cap \bigcap_{i \in [m] }H^-(F_i, r_i)$.
Now, $\sigma$ is in $\calT_G(\Delta_{k,n})$ and the claim follows from Lemma~\ref{lm:CutsDoNotInterfere} which yields
\[
 \sigma \cap H_G = \sigma \cap H_G \cap \bigcap_{i \in [m] }H^-(F_i, r_i) = \sigma^N \cap H_G = \tau. 
\]

Second claim: 
Since $\tau\in \calT_G(\Delta_{k,n})$, there exists a face
$\sigma \prec\Delta_{k,n}$ such that $\sigma\cap H_G=\tau$. By assumption, there exist two points $q^+, q^-\in \sigma$ such that $q^+\in (H_G^+)^\circ \cap \sigma$, likewise for $-$. Moreover, we know that $q^+ \in \sigma^N \coloneqq \sigma \cap \bigcap_{i \in [m] }H^-(F_i, r_i) $, otherwise we would have two $+$ in $\sgn q^+$. This implies that we cannot have a point $p\in \sigma^N$ with all negative signs, because otherwise it would be in the interior of $\sigma^N$ and $\tau$ would split $\sigma^N$, which contradicts the assumption $\tau\notin \calT_G(N) \cap H_G$. By Lemma~\ref{lm:FacesInTheError}, this happens if and only if for exactly one $F_i$ we have that $\sigma \cap H_G \subset H_{F_i}$.
By Lemma~\ref{lemma:TwoCuts}, we get that $(F_i, G)$ is a modular pair, which finishes the claim.

\end{proof}

\begin{ex}
\label{ex:M1M2M3}
    To see the role played by modular pairs,
    consider the following rank-4 matroids:
    \begin{enumerate}
        \item $M_1$ with nonbases $\{\{1,2,3,4\}, \{1,2,5,6\}\}$ on $E = [8]$,
        \item $M_2$ with nonbases $\{\{1,2,3,4\}, \{1,5,6,7\}\}$ on $E = [8]$,
        \item $M_3$  with nonbases $\{\{1,2,3,4\}, \{5,6,7,8\}\}$ on $E = [8]$.
    \end{enumerate}
    Each has two cyclic flats $F, G$, the nonbases, which are also hyperplane-circuits. 
    Thus, we only need to compute the \texttt{cd}-index of the cuspidal matroid $\Lambda_{1,4,4,8}$, 
    the hypersimplex $\Delta_{4,8}$ and the error term in the case of modular pair present in $M_1$.  
    In $M_2$ and  $M_3$ we have $\card{F\cap G}=1$ 
    and $\card{F\cap G}=0$, respectively, so they are non-isomorphic. 
    We have:
\begin{align*}
  \Psi_{cd}(M_1) = &3664cd^3 + 2432dcd^2 + 2748d^2cd + 3428d^3c + 456c^3d^2 + 1550c^2dcd +1834c^2d^2c \\
&\hspace{0.2cm}+ 1768cdc^2d + 3616cdcdc + 2664cd^2c^2 + 408dc^3d + 1300dc^2dc + 1816dcdc^2 \\
&\hspace{0.2cm}+ 1034d^2c^3 + 16c^5d + 110c^4dc + 376c^3dc^2 + 633c^2dc^3 + 460cdc^4 + 66dc^5 + c^7, \\
\Psi_{cd}(M_2) &= \Psi_{cd}(M_3) = \\&3664cd^3 + 2432dcd^2 + 2752d^2cd + 3432d^3c + 456c^3d^2 + 1552c^2dcd  + 1836c^2d^2\\
&\hspace{0.2cm} + 1768cdc^2d + 3616cdcdc + 2664cd^2c^2 + 408dc^3d + 1300dc^2dc + 1816dcdc^2 \\
&\hspace{0.2cm} + 1036d^2c^3 + 16c^5d + 110c^4dc + 376c^3dc^2 + 634c^2dc^3 + 460cdc^4 
+ 66dc^5 + c^7. \qedhere 
\end{align*}
\end{ex}

\subsection{The computational data}
    \label{sub:ComputationalData}
Now we consider a modular pair $F,G$ in a connected split matroid $M$ 
and analyse $W_{F,G}$ computationally
to understand  which statistics of $F,G$ are needed.

\begin{lemma} 
\label{lm:WDependsOnlyOnAlphaBetaAB}
 Let $(F,r), (G,s) \in \properZ(M)$ be a modular pair of a connected split matroid $M$.
The error $W_{F,G}$ depends only on 
$\alpha = r - \card{F \cap G}$, 
$\beta = s - \card{F \cap G}$,
$a = \card{F \setminus G}$, and
$b = \card{G \setminus F }$. 
\end{lemma}

\begin{proof} 
    This follows from the fact that the error term $\calE$ is evaluated on $\sigma \cap H_G$ for $\sigma \in \calT_{G}(\sigma)$ that satisfy $\sigma \cap H_G \subset H_F$; 
    and that by Lemma~\ref{lemma:TwoCuts} this polytope $\sigma \cap H_G = \sigma \cap H_G \cap H_F = Q$ depends only on $(a, b; \alpha, \beta)$.
\end{proof}

\begin{lemma} 
\label{lm:ModularPairStatistics}
 Let $(F,r), (G,s) \in \properZ(M)$ be a modular pair in a connected split matroid $M$, 
 and $(a, b; \alpha, \beta)$ as in Lemma~\ref{lm:WDependsOnlyOnAlphaBetaAB}.
The error $W_{F,G}$ of $\Psi_{cd}$ equals
\begin{align*} 
W_{F,G} =
 \sum_{p=1}^{\alpha}
\sum_{q=1}^{\beta}
\sum_{i=p+1}^{a-\alpha+p}
\sum_{j=q+1}^{b-\beta+q}
\binom{a}{i}\binom{b}{j}\binom{a-i}{\alpha-p}\binom{b-j}{\beta-q}
\Psi_{cd}(\Delta_{p,i}\times\Delta_{q,j})\cdot d
\cdot\Psi_{cd}(\Delta_{1,n-i-j}).
\end{align*}
\end{lemma}

\begin{proof} 
 To get a face in $\calS(F,G)$ from Theorem~\ref{thm:CDIndexSplitMatroid}: 
choose three pairwise disjoint sets $A, B, C$ such that $A \subset F \setminus G$, $B \subset G \setminus F$ and $F\cap G \subset C \subset F\cup G$. 
Write $p = \rk(F) - \card{C \cap F}$, $q = \rk(G) - \card{C \cap G}$ and $D=[n]\setminus (A\cup B \cup C)$. 
This choice corresponds to a face given by
\[
\Delta_{\card{C}, C} \times \Delta_{p, A} \times \Delta_{q, B} \times \Delta_{0, D}.
\]
By Lemma~\ref{lm:SplittingFace} we want to avoid the faces of  $\Delta_{\rk(G), \card{G}} \times \{v\}$ or $\{w\} \times \Delta_{k-\rk(G), n-\card{G}}$ for vertices $v$ of $\Delta_{k-\rk(G), n-\card{G}}$ and $w$ of $\Delta_{\rk(G), \card{G}}$. 
So we need that $\card{A}> p >0$ and $\card{B}> q >0$. 
Once we have this characterization we can evaluate the sum over all the faces of $\calS$ by grouping the faces with the same combinatorial properties. 
Notice that for the sake of evaluating the \texttt{cd}-index the only relevant variables are: $\alpha$, $\beta$, $i\colon=\card{A}$ and $j\colon=\card{B}$. 
If we fix these four values, there are exactly 
$C_{p,q,i,j}=
\binom{\card{F \setminus G}}{i}
\binom{\card{G\setminus F}}{j}
\binom{\card{F \setminus G}-i}{\rk(F)- p -\card{F\cap G}}\binom{\card{G\setminus F}-j}{ \rk(G)- q - \card{F\cap G}}$ possible faces of this combinatorial type. 
Hence, we obtain that 
$\sum_{\sigma \in \calS(F,G)} \Psi_{cd}(\sigma) \cdot d \cdot \Psi_{cd}(\Delta_{s, \card{G}} \times \Delta_{k-s, n-\card{G}}/\sigma)$ can be rewritten as: 
\begin{align*}
\sum_{p=1}^{\rk(F)-\card{F\cap G}}
\sum_{q=1}^{\rk(G)-\card{F\cap G}}
\sum_{i=p+1}^{\card{F}-\rk(F)+p}
\sum_{j=q+1}^{\card{G}-\rk(G)+q}
C_{p,q,i,j}
\Psi_{cd}(\Delta_{p,i}\times\Delta_{q,j})\cdot d
\cdot\Psi_{cd}(\Delta_{1,n-i-j}),
\end{align*}
where we are using that the face-figure of a product of hypersimplices is simplicial for any face of dimension higher than $2$.  
\end{proof}

This shows that the error term is symmetrical under swapping $F$ and~$G$. 
For rewriting the general formula, let $\lambda(r,h)$ be the number of proper non-empty cyclic flats of $M$ of rank $r$ and size $h$,
and $\mu( a,b; \alpha,\beta)$ the number of modular pairs of cyclic flats $F, G$ with prescribed $\alpha$, $\beta$, $a$ and  $b$, as defined in Lemma~\ref{lm:WDependsOnlyOnAlphaBetaAB}.
Moreover, since Lemma~\ref{lm:ModularPairStatistics} shows that $W_{F,G}$ depends only on $\mbfr = (a, b; \alpha, \beta)$, we can write $W(\mbfr)$.

\begin{theorem} 
    \label{thm:RecursiveCD}
    Let $\scrP(M) = \Delta_{k,n} \cap \bigcap_{(F, r) \in \properZ(M)} H^-(F, r)$ be connected split.
    We have that
\begin{align}
\label{eq:recursivecd}
\Psi_{cd}(M) =& \Psi_{cd}(\Delta_{k,n}) + \sum_{0 < r < h < n} \lambda(r,h) \left( \Psi_{cd}(\Lambda_{k-r,k,n-h,n}) - \Psi_{cd}(\Delta_{k,n}) \right) \\
&- \sum_{\mbfr} \mu(\mbfr) \cdot W(\mbfr).
\end{align}
Here the second sum ranges over all $\mbfr = ( a,b;\alpha,\beta)$ such that $0 < \alpha < a < n$, $0 < \beta < b < n$ and $(\alpha,a)\leq (\beta,b)$ with respect to the lex-order.
\end{theorem}

\begin{proof}
    The result follows from Theorem~\ref{thm:CDIndexSplitMatroid},
    the invariance of $\Psi_{cd}$ under matroid isomorphism, 
    and Lemma~\ref{lm:WDependsOnlyOnAlphaBetaAB}.
\end{proof}

So indeed, to compute the \texttt{cd}-index the only information needed is the numbers $\mu( a,b; \alpha,\beta)$
and knowing the result for hypersimplices and cuspidal matroids.
We keep the results of these computations in a database to speed up the running time.
We use the product formula from Equation~\eqref{eq:ProductFormula} inside the correction~$\calE$. 

\begin{remark}
  Kim derived a formula for the \texttt{cd}-index of $\scrP(M) \subset \Delta_{2,n}$
in terms of \texttt{cd}-indices of base polytopes of rank-2 matroids with $1$, $2$ and $3$ parallelism classes \cite[Section~5]{kim10}. 
Furthermore, for the first two cases a formula is provided for the calculations, but the case with $3$ parallelism classes remained open. 
An application of~Theorem~\ref{thm:CDIndexSplitMatroid} gives a recursion for them. 
  In this particular case all the hypersimplices that appear in the formula are simplices, hence to speed up the calculation we use a result from Purtill~\cite{pur93} proving that the \texttt{cd}-index of $\Delta_{1,n}$ is the $n$th André polynomial. 
  In~\cite{fv26} we provide code that recomputes Table~1 of \cite{kim10}.
\end{remark}

\subsection{Almost-valuative invariants and the  \texorpdfstring{$f$-vector}{f-vector}}
  \label{sub:AlmostValuative}
Note that the proof of Theorem~\ref{thm:CDIndexSplitMatroid} does not use any specific properties of the correction~$\calE(\sigma \cap H)$ beyond those encapsulated in the following definition:

\begin{definition} 
    \label{def:AlmostValuative}
Let $\Psi : \operatorname{MatPoly} \to \mathcal{G}$ 
be a function, and $(\mathcal{G}, +, \cdot)$ an arbitrary ring.
It is \emph{almost-valuative} if for some function $\calE : \operatorname{MatPoly} \to \mathcal{G}$ and constant $K \in \mathcal{G}$ we have:
for all matroids $M$ and hyperplanes $H$ splitting $\scrP(M)$ into two matroid polytopes 
\begin{align}
    \label{eq:AlmostValuative}
\Psi(M) =& \Psi(\scrP(M) \cap H^+) + \Psi(\scrP(M) \cap H^-) - \Psi(\scrP(M) \cap H) \cdot K 
                 \\ \nonumber &- \sum_{\sigma \in \calT_H(M)} \calE(\sigma \cap H).
\end{align} 
\end{definition}

We obtain immediately our main theorem, stated as Theorem~\ref{thm:AlmostValuativeInvariantsIntro} in the intro, which we now repeat for the convenience of the reader. 

\begin{theorem}  
    \label{thm:AlmostValuativeInvariants}
    Let $M$ be a connected split matroid such that $\properZ(M) = \aset{(F_i,r_i)}_{i \in [m]}$, 
    and $\Psi$ an almost-valuative invariant with correction~$\calE$.
    We have that
    \[ \Psi(M) = \Psi(\Delta_{k,n}) + \sum_{F \in \properZ(M)} (\Psi(\Lambda^F_M) - \Psi(\Delta_{k,n}) )- \sum_{(F,G) \in \calM \calP(M)} W_{F,G}.\]
Here $\calM \calP(M)$ is the set of modular pairs $(F_i, F_j)$ of cyclic flats with $i < j$, and
\begin{align}
    W_{F,G}=\sum_{\sigma \in \calS(F,G)} \calE(\sigma \cap H).
\end{align}
Here, $\calS(F,G)$ are faces $\sigma$ in $\calT_G(\Delta_{k,n})$ such that $\sigma \cap H_G \subset H_F$.

Moreover, we have that $W_{F,G}$ depends only on $a = \card{F \setminus G}$, $b = \card{G \setminus F}$, $\alpha = \rk(F) - \card{F \cap G} $ and $\beta = \rk(G) - \card{F \cap G}$.
Also
\begin{align*} 
\Psi(M) =& \Psi(\Delta_{k,n}) 
    + \sum_{0 < r < h < n} \lambda(r,h) \left( \Psi(\Lambda^{r,h}_{k,n}) - \Psi(\Delta_{k,n}) \right) \\
&- \sum_{\mbfr} \mu(\mbfr) \cdot W(\mbfr).
\end{align*}
The second sum ranges over all $\mbfr = ( a,b;\alpha,\beta)$ 
such that $0 < \alpha < a < n$, $0 < \beta < b < n$ and $(\alpha,a)\le_{\operatorname{lex}} (\beta,b)$ in the lexicographic order.
\end{theorem}

As an example, we consider the $f$-vector $(f_i)$, where $f_i$ counts the number of $i$-th dimensional faces of a polytope $P$. 
Encoding it as a polynomial in $\ZZ[t]$ we get that it is almost-valuative.

\begin{lemma} 
    \label{lemma:fVectorIsAlmostValuative}
    The following invariant is almost-valuative:
\begin{align} 
    \label{eq:fPolynomial}
    \Psi_f(P) = \sum_{i = 0}^{\dim(P)} f_i t^i \subset \ZZ[t].
\end{align}
In particular, for a hyperplane $H$ we have
\begin{align}
\Psi_{f}(P) =& \Psi_{f}(P^+) + \Psi_{f}(P^-) - \Psi_{f}(P \cap H)  
\\ \nonumber &- \sum_{\sigma \in \calT_{H}(P)} t^{\dim(\sigma \cap H)} \cdot (1+t).
\end{align} 
\end{lemma}

\begin{proof} 
   A face $\sigma$ of $P$ in $\calT_H(P)$ is counted twice by $\Psi_{f}(P^+) + \Psi_{f}(P^-)$, 
   once as $\sigma \cap H^+$ and once as $\sigma \cap H^-$.
   The correction $t^{\dim(\sigma \cap H)} \cdot t$ handles this by cancelling out one occurrence.
   Also, $\sigma$ generates a face not present in $P$ which is counted twice by $\Psi_{f}(P^+) + \Psi_{f}(P^-)$; one occurrence is corrected by one term in $\Psi_{f}(P \cap H)$ and the other by the correction $t^{\dim(\sigma \cap H)} \cdot 1$.
   Finally, a face $\tau$ of $P$ contained in $H$ is counted twice, but one occurrence cancels out with the remaining terms of~$\Psi_{f}(P \cap H)$. 
\end{proof}

\begin{remark} 
    Equation~\ref{eq:fPolynomial} is implicitly proven in \cite[Proposition~2.7]{fs25}.
    This means that Theorem~\ref{thm:AlmostValuativeInvariants} applied to $\Psi_f$ recovers the results of \cite{fs25}.
\end{remark}
 
\section{Applications and future directions}
    \label{sec:ApplicationsAndFutureDirections}

\subsection{Sparse paving matroids}
  \label{sub:SparsePavingMatroids}
Sparse paving matroids are connected split matroids where all proper cyclic flats are circuit hyperplanes, i.e.~rank $k-1$ and size $k$.
By \cite[Lemma 2.8]{fer22} this is equivalent to the more familiar definitions.
Two cyclic flats form a modular pair if and only if the intersection has $k-2$ elements.
This class includes many famous examples like the Fano matroid, $M(K_4)$ and Vamos matroid. 
This simplifies Equation~\eqref{eq:recursivecd} to the following.

\begin{corollary}
Let $M$ be a connected sparse paving matroid of rank $k$ on $n$ elements having
exactly $\lambda$ circuit-hyperplanes, and let $\mu$ count the pairs of circuit-hyperplanes which have $k-2$
elements in common. Then
\[
\Psi_{cd}(M) =\lambda  \Psi_{cd}(\Lambda_{1,k,n-k,n}) - (\lambda-1)\Psi_{cd}( \Delta_{k,n}) -  \mu (c^2d+2d^2) \Psi_{cd}(\Delta_{1,n-4}).
\]
\end{corollary}

This formula shows that $\lambda$ and $\mu$ are the only two matroidal data needed to encode the flag $f$-vector of a sparse paving matroid.

\begin{ex}
  Let $F$ be the Fano matroid and $V$ be the Vamos matroid, then:
\begin{align*}
  \Psi_{cd}(F) &= 7  \Psi_{cd}(\Lambda_{1,3,4,7}) - 6\Psi_{cd}( \Delta_{3,7}) -  21 (c^2d+2d^2)\cdot \Psi_{cd}(\Delta_{1,3}) \\
  &=  364cdcd + 98cdc^3 + 490cd^2c + 145c^2dc^2 + 221c^2d^2 + 91c^3dc + 19c^4d + c^6 \\
  & \hspace{0.5cm} + 462dcdc + 186dc^2d + 26dc^4 + 298d^2c^2 + 482d^3,\\
  \Psi_{cd}(V) &= 5  \Psi_{cd}(\Lambda_{1,4,4,8}) - 4\Psi_{cd}( \Delta_{4,8}) -  8 (c^2d+2d^2)\cdot \Psi_{cd}(\Delta_{1,4}) \\
  &= 3580cdcdc + 1690cdc^2d + 415cdc^4 + 2670cd^2c^2 + 3700cd^3 + 1596c^2dcd \\
  & \hspace{0.5cm} + 635c^2dc^3 + 1922c^2d^2c + 415c^3dc^2 + 510c^3d^2 + 131c^4dc + 19c^5d + c^7 \\
  & \hspace{0.5cm} + 2020dcdc^2 + 2720dcd^2 + 1438dc^2dc + 432dc^3d + 63dc^5 + 2984d^2cd \\
  & \hspace{0.5cm} + 1098d^2c^3 + 3772d^3c.  \qedhere 
\end{align*}
\end{ex}
\subsection{Code for computation}
  \label{sub:CodeForComputation}
We implement the recursions of Proposition~\ref{prop:PhiForHypersimplex} and Lemma~\ref{lm:PhiForCuspidal} in Python and in SageMath. 
We provide a function in SageMath that takes any matroid as input, checks if the matroid is connected split, 
and if yes then returns its \texttt{cd}-index. 
The code for this can be found in: \\
\url{https://github.com/AV-2/cd-index-matroid-base-polytope} 
 
\subsection{Formulas for other polytopes}
We generalize to the \texttt{cd}-index setting questions posed by  Ferroni and Schr\"oter in \cite[Section~3]{fs25}.
The independence polytope $\scrP_{\text{ind}}(\mathrm{M})$, 
is defined as the convex hull of the indicator vectors of all independent sets of $M$. 
Note that $\mathcal{P}(\mathrm{M})$ appears as a facet of $\mathcal{P}_{\text{ind}}(\mathrm{M})$,
so a natural challenge is to extend our enumerative results to this larger class of polytopes.
\begin{problem}
  Find a formula for the \texttt{cd}-index of the independence polytope for (connected) split matroids.
\end{problem}
We believe that a first step would be to establish whether the cuts of $\scrP(M)$ can be extended to define $\scrP_{ind}(M)$.
If successful, then the next step is studying whether the geometry of Section~\ref{sec:TheGeometryOfSplitConnectedMatroids} still holds.

On the other hand, it is natural to ask whether these formulas can be generalized to arbitrary matroids. 
We pose the following question, which mirrors similar inquiries into the $f$-vectors of arbitrary matroid polytopes.

\begin{question}
  Can the approach of splitting polytopes be used to obtain a computable formula for the \texttt{cd}-index of an arbitrary matroid polytope? Independently of the details of the computation, what is the precise matroidal data needed to recover the \texttt{cd}-index?
\end{question}

All matroid polytopes are cut out of the hypersimplex (or the unit cube) by compatible split hyperplanes. Since the behavior of the \texttt{cd}-index under hyperplane splits is well-understood (see e.g., \cite{kim10}), it seems reasonable to attempt to generalize these techniques to other classes of $0/1$-polytopes.

\begin{problem}
  Describe the \texttt{cd}-index of $0/1$-polytopes in terms of their supporting split hyperplanes.
\end{problem}

\subsection{Inequalities and structural properties}
A major theme in the study of the \texttt{cd}-index is the non-negativity of its coefficients \cite{pur93}. 
 For specific classes of polytopes, stronger inequalities may hold.

\begin{question}
  Do the coefficients of the \texttt{cd}-index of matroid base polytopes satisfy any specific inequalities or structural properties beyond general non-negativity?
\end{question}

As we have seen, some enumerative invariants of split matroids, specifically the $f$-vector and the \texttt{cd}-index of the base polytope, are fully determined by the parameter sets $\{\lambda_{r,h}\}$ and $\{\mu(a,b;\alpha,\beta)\}$. 
Since the former encodes the flag $f$-vector, it is natural to ask whether the enumerative information contained in the flag $f$-vector is strictly richer than that of the ordinary $f$-vector for this class of matroids. More precisely, we ask the following natural question that arises from our work.

\begin{problem}
  Is the $f$-vector of a connected split matroid sufficient to determine its flag $f$-vector?
\end{problem}

We have verified this affirmatively via computer experiments for all connected split matroids on a ground set of size at most $12$, as well as for all rank $2$ connected split matroids on a ground set of size at most $50$. 

For sparse paving connected matroids, we can also answer this in the affirmative in general. By examining the first few entries of the $f$-vector, specifically the number of vertices $f_0$ and $2$-faces $f_2$, we are able to recover the defining parameters $\lambda$ and $\mu$. Consequently, since $\lambda$ and $\mu$ determine the \texttt{cd}-index, the ordinary $f$-vector already encodes the full flag enumerative information in this setting. 

However, one must be careful not to overgeneralize this property. It might intuitively seem possible to uniquely recover the underlying matroidal data (the $\lambda$'s and $\mu$'s) directly from the $f$-vector for \emph{all} split matroids, just as we did for sparse paving matroids. This is actually false. Consider two split matroids of rank $k=4$ on $n=8$ elements defined by the following proper cyclic flats:
\begin{itemize}
  \item $\mathbf{M_1}$: $F_1 = \{1,2,3,4,5\}$ (rank $3$) and $F_2 = \{5,6,7,8\}$ (rank $3$).
  \item $\mathbf{M_2}$: $F_1 = \{1,2,3\}$ (rank $2$) and $F_2 = \{4,5,6,7\}$ (rank $3$).
\end{itemize}

In both cases, the pairs of cyclic flats fail the modularity property, meaning their respective $\mu$ parameters coincide. Their $\lambda$ parameters clearly differ due to the varying ranks and sizes of the flats. Yet, despite possessing different $\lambda$ and $\mu$ data, these two matroids share the exact same $f$-vector and \texttt{cd}-index. Therefore, while the $f$-vector might determine the flag $f$-vector, it cannot generally be used to reconstruct these data of the matroid.

\section{Acknowledgements}

We thank Luis Ferroni for suggesting the problem, and for his encouragement and feedback during the realization of this project. 
We thank Damiano Testa for discussions, proofreading, and suggestions for improvements.
We thank Benjamin Schröter and two anonymous referees who reviewed an extended abstract for pointing out to further literature and open problems.
This project started during the ``VII winter school Geometry, Algebra and Combinatorics of Moduli Spaces and Configurations'', in Dobbiaco.
We thank the organizers for the wonderful event.
TF was supported by the Warwick Mathematics Institute Centre for Doctoral Training, the Heilbronn Institute for Mathematical Research (HIMR), and the UK Engineering and Physical Sciences Research Council (EPSRC) grant number EP/V521917/1.
AV was supported by EPSRC grant number EP/X02752X/1.

\printbibliography

\end{document}